\documentclass[11pt]{amsart}
\usepackage{amsfonts, amstext, amsmath, amsthm, amscd, amssymb, upgreek, wasysym}
\usepackage{graphicx, color,  subfigure, wrapfig, overpic, url}
\usepackage[dvipsnames]{xcolor}
\usepackage{tikz, url}
\usepackage[margin=1.3in]{geometry}
\usepackage{graphics}
 
\AtBeginDocument{
   \def\MR#1{}
}

\newcommand{\R}{\mathbb{R}}

\newcommand{\Z}{\mathbb{Z}}
\renewcommand{\H}{\mathbb{H}}
\renewcommand{\P}{\mathcal{P}}
\newcommand{\PP}{\mathtt{P}}
\newcommand{\vol}{{\rm vol}}
\newcommand{\voct}{{v_{\rm oct}}}
\newcommand{\vtet}{{v_{\rm tet}}}
\newcommand{\LL}{\mathcal{L}}

\newcommand{\K}{\upkappa}
\renewcommand{\setminus}{{\smallsetminus}}
\newcommand{\volbp}{{\rm vol}^{\lozenge}}
\newcommand{\G}{\mathcal G}
\newcommand{\vg}{\nu^{\lozenge}(\mathcal G)}
\newcommand{\vgs}{\nu^{\lozenge}(\G_s)}
\newcommand{\zfd}{z^{\rm fd}}
\newcommand{\vbar}{\overline{\nu}}
\newcommand{\volp}{{\rm vol}^{\perp}}
\newcommand{\gb}{G^b}
\newcommand{\vgb}{\vbar(\G)}
\newcommand{\vgbs}{\vbar(\mathcal G_s)}
\newcommand{\Gb}{\mathcal G^{b}}
\newcommand{\m}{\mathrm{m}}

\newtheorem{theorem}{Theorem}
\newtheorem{corollary}[theorem]{Corollary}
\newtheorem{lemma}[theorem]{Lemma}
\newtheorem{prop}[theorem]{Proposition}

\newtheorem{conjecture}[theorem]{Conjecture}

\newtheorem*{namedtheorem}{\theoremname}
\newcommand{\theoremname}{testing}

\theoremstyle{definition}
\newtheorem{definition}[theorem]{Definition}
\newtheorem{question}[theorem]{Question}
\newtheorem{remark}[theorem]{Remark}

\title[Geometric bounds for spanning tree entropy of planar lattice graphs]
{Geometric bounds for spanning tree entropy\\ of planar lattice graphs}
\author[A.\ Champanerkar]{Abhijit Champanerkar}
\address{Department of Mathematics, College of Staten Island \& The Graduate Center, City University of New York, New York, NY}
\email{abhijit@math.csi.cuny.edu}
\author[I. \ Kofman]{Ilya Kofman}
\address{Department of Mathematics, College of Staten Island \& The Graduate Center, City University of New York, New York, NY}
\email{ikofman@math.csi.cuny.edu}

\thanks{May 8, 2025}

\begin{document}

\maketitle

\begin{abstract}
  We prove infinitely many cases of conjectured sharp upper and lower
  bounds for the spanning tree entropy of any planar lattice graph.  These
  bounds come from volumes of associated hyperbolic alternating links,
  right-angled hyperbolic polyhedra and hyperbolic regular ideal
  bipyramids.  For many planar lattice graphs, we show these bounds are easy
  to compute and provide excellent numerical estimates for the
  spanning tree entropy.
\end{abstract}

\section{Introduction}

Let $\G$ be a connected locally finite planar biperiodic graph, which
is invariant under the action of a lattice $\Lambda\cong\Z^2$.
We will call $\G$ a planar lattice graph.  Consider the exhaustive
nested sequence of finite connected planar graphs $\Gamma_n$ given by
$n\times n$ copies of the $\Lambda$-fundamental domain of $\G$.  Let
$\tau(\Gamma_n)$ be the number of spanning trees of $\Gamma_n$, and
let $V\Gamma_n$ be its vertex set.  The \emph{spanning tree entropy
  of} $\G$ is defined as
$$ z_{\G} = \lim_{n\to\infty}\frac{\log\tau(\Gamma_n)}{|V\Gamma_n|}.$$

In 1973, Temperley \cite{Temperley} computed $z_{\G}$ for the square
grid using a bijection between spanning trees and dimer coverings
(perfect matchings).  Temperley's bijection has been extended in many
ways to compute $z_{\G}$ in more general settings (see, e.g.,
\cite{KOS, KPW}).  For planar lattice graphs, $z_{\G}$ is sometimes
exactly computable using its quotient graph on the torus (see, e.g.,
\cite{ckl:mm_voldet, KOS, Lyons}).

A surprising fact about $z_{\G}$ for many planar lattice graphs is that its
value is closely related to hyperbolic geometry.  Let $\vtet\approx
1.01494$ be the volume of the hyperbolic regular ideal tetrahedron, and
$\voct\approx 3.66386$ be the volume of the hyperbolic regular ideal
octahedron.  If $\G_{\triangle},\,\G_{\square}$ and $\G_{\hexagon}$
denote the regular triangular, square and hexagonal lattice graphs, then
$$ 2\pi\,z_{\G_{\triangle}} =  10\vtet, \qquad 2\pi\,z_{\G_{\square}} =   2\voct, \qquad 2\pi\,z_{\G_{\hexagon}} =  5\vtet.$$
See, e.g., \cite[Theorems~12, 13]{ckl:mm_voldet}; see Section~\ref{sec:rak} for another proof.
In Conjecture~\ref{conj:vol-entropy} below, we propose upper and lower
bounds for $z_{\G}$ for any planar lattice graph $\G$.  The volume of an
associated hyperbolic link provides the lower bound, which is sharp
for the regular planar lattice graphs.

There is a well-known correspondence between alternating link diagrams
and planar graphs: If the faces of the link diagram are checkerboard
colored, the {\em Tait graph} is the planar graph for which a vertex
is assigned to every shaded region and an edge to every crossing of
the link diagram.  Using the other checkerboard coloring yields the
dual graph.  Conversely, the medial graph of any planar graph (or its
dual) is the projection graph of an alternating link diagram.

Similarly, a planar lattice graph $\G$ is the Tait graph of a
\emph{biperiodic alternating link} $\LL$ in $\R^2\times I$.  Then
$G=\G/\Lambda$ is a graph on the torus $T^2$, which is the Tait graph
of an alternating link diagram on $T^2$.  Let $L$ be the link in the
thickened torus $T^2\times I$, such that $L=\LL/\Lambda$.  Let $G^*$
be the dual graph of $G$ on $T^2$.  Let $VG, EG, FG$ denote the sets
of vertices, edges and faces of $G$.  If $G$ and $G^*$ are
$2$-connected and their faces are topologically disks on the torus,
then $(T^2\times I)-L$ is a complete finite-volume hyperbolic
$3$--manifold \cite[Theorem~4.2]{HowiePurcell}.
Moreover, the crossing number $c(L)=|EG|$ is minimal among all
toroidal diagrams of $L$ \cite{Adams:crossings}.
Using its hyperbolic volume, we define
$$ \vol(G)=\vol((T^2\times I)-L) \quad \text{and} \quad \vol(\G)=\frac{\vol(G)}{|VG|} \quad \text{and} \quad \vgb=\frac{|EG|\voct}{|VG|}.$$

\begin{conjecture}\label{conj:vol-entropy}
$$ \vol(\G) \ \leq \ 2\pi\,z_{\G} \ \leq \ \vgb.$$
\end{conjecture}

Conjecture~\ref{conj:vol-entropy} lies at the intersection of
hyperbolic geometry, number theory, probability and graph theory.  In
this paper, we explain the context and provide numerical evidence for
Conjecture~\ref{conj:vol-entropy}, and we prove infinitely many cases
of Conjecture~\ref{conj:vol-entropy}. Most of the results proved in
this paper address the lower bound in
Conjecture~\ref{conj:vol-entropy}.  The hyperbolic volume of a link
can be computed numerically using the computer program SnapPy
\cite{snappy}.  Finding the spanning tree entropy of a planar lattice
graph involves laborious computations and many recent published papers
just compute examples.  Finding exact values for both quantities is
far more difficult, and when possible requires advanced techniques
from number theory; see e.g., \cite{ckl:mm_voldet}.  Instead, below we
define several new geometric invariants associated to planar graphs
and lattice graphs that are easier to compute, are numerically very
close to the spanning tree entropy in many examples, and are easier to
study asymptotically.

In Section~\ref{sec:bipyramid}, we define the bipyramid volume $\vg$ of a
planar lattice graph $\G$, and we show
$$ 0 \ < \ \vol(\G) \ \leq \ \vg.$$
If $\vg \leq 2\pi\,z_{\G}$, then $\vol(\G) \leq 2\pi\,z_{\G}$, as in
Conjecture~\ref{conj:vol-entropy}.  The bipyramid volume is easier to
compute than $\vol(\G)$ or $z_{\G}$, and we show for 16 planar
lattice graphs that it provides both a lower bound and an excellent estimate
for $z_{\G}$.  We also show that for one planar lattice graph, the bipyramid
volume is not a lower bound for $z_{\G}$, which contradicts
\cite[Conjecture~1]{ckl:mm_voldet}.  This exceptional planar lattice graph
still satisfies Conjecture~\ref{conj:vol-entropy}.

In Section~\ref{sec:planar}, we recall the Vol-Det Conjecture
\cite[Conjecture~1.10]{ckp:gmax} from hyperbolic knot theory, which is
that for any alternating hyperbolic link $K$,
$$ \vol(K)<2\pi\,\log\det(K).$$  We present strict inequalities for
finite alternating links that extend to the sharp inequalities for
biperiodic alternating links in Conjecture~\ref{conj:vol-entropy}.
Namely, we define the volume and bipyramid volume for a finite planar
graph, and we state Conjecture~\ref{conj:planar}, which gives similar
geometric bounds for the number of spanning trees of finite planar
graphs.  For finite planar graphs that asymptotically converge to the
planar lattice graph $\G$, their bipyramid volumes converge to that of
$\G$.  We show that whenever the bipyramid volume of $\G$ is a lower
bound for $z_{\G}$ with a strict inequality, we obtain infinitely many
finite planar graphs converging to $\G$ that satisfy
Conjecture~\ref{conj:planar}, which provide infinite families of
alternating links that satisfy the Vol-Det Conjecture.

In Section~\ref{sec:infinite}, we apply another useful property of the
bipyramid volume, that its logarithmic growth rate is similar to that
of $z_{\G}$.  We prove Conjecture~\ref{conj:vol-entropy} for
infinitely many cases using three different ways to obtain infinite
families of planar lattice graphs: by inserting parallel edges, by
truncating any $3$-regular planar lattice graph, and by taking the medial
graph of a $3$-regular planar lattice graph.  The bipyramid volume provides
a lower bound for $z_{\G}$ in all of these infinitely many cases.
Previously, in \cite{ckl:mm_voldet}, it was proved that the bipyramid
volume provides a lower bound for $z_{\G}$ for six biperiodic
alternating links using rigorous computations for the Mahler measures
of the corresponding 2-variable polynomials.  In \cite{Gan1}, four
other lower bounds were proved using methods from graph theory in
\cite{TeuflWagner}.

In Section~\ref{sec:rak}, we define another geometric invariant of a
planar lattice graph $\G$, the volume of an associated hyperbolic
right-angled polyhedron, which we call the right-angled volume
$\volp(\G)$.  If $\G$ and $\G^*$ satisfy a geometric condition called
orthogonal duality, then
$$ \volp(\G) \ \leq \ \vol(\G) \ \leq \ \vg.$$
Thus, Conjecture~\ref{conj:vol-entropy} implies $\volp(\G)\leq 2\pi\,z_{\G}$,
and we show that $\volp(\G)$ is exactly computable using the
local geometry of $G$.
The three regular planar lattice graphs satisfy both isoradiality and
orthogonal duality, which we exploit using the isoradial dimer model
to prove that the lower bound in Conjecture~\ref{conj:vol-entropy}
holds with equality for these lattice graphs.

\subsection*{Acknowledgements}
We thank Hong-Chuan Gan for stimulating discussions.
We also acknowledge support from the Simons Foundation and PSC-CUNY.

\begin{figure}
\begin{tabular}{cc}
\includegraphics[height=1.45in]{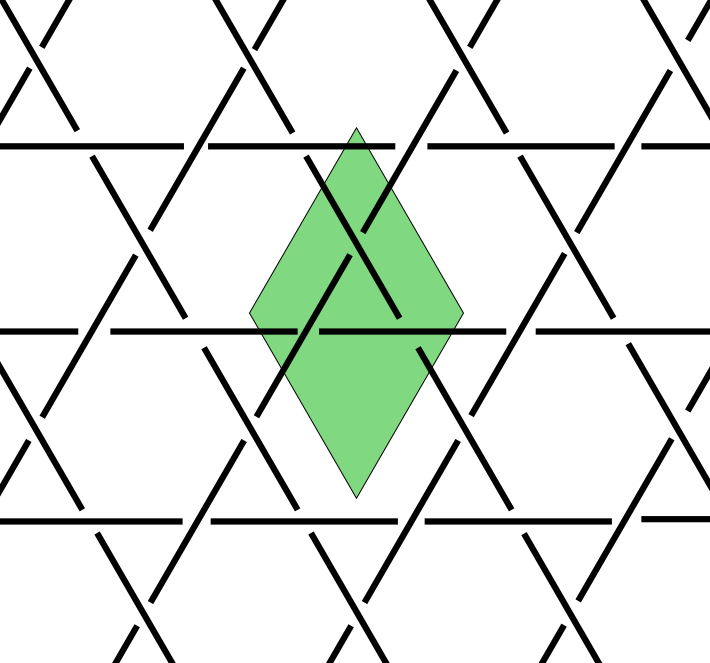} &
\qquad
\includegraphics[height=1.45in]{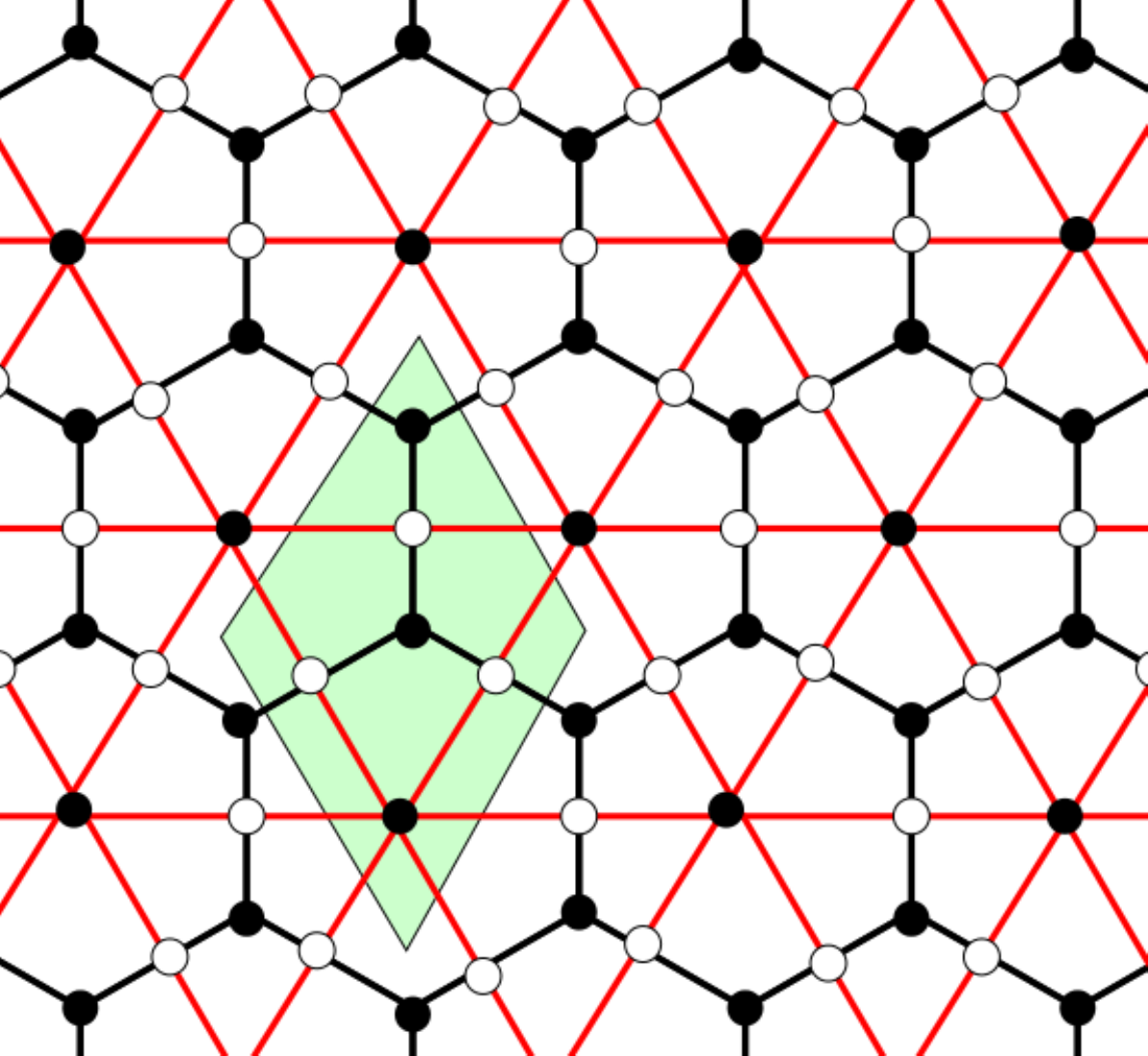}
\\
(a) & \qquad (b) \\
\end{tabular}
\caption{(a) Biperiodic triaxial link $\LL$, and fundamental domain
  for $L$.  (b) Temperleyan graph $\Gb$, and fundamental domain for
  $\gb$.}
\label{fig:triaxial}
\end{figure}

\section{Bipyramid volume of a planar lattice graph}\label{sec:bipyramid}

Finding numerical values for the spanning tree entropy of a planar
lattice graph involves laborious computations, and exact values are usually
difficult to prove (see \cite{ckl:mm_voldet}).  Instead, below we
define the bipyramid volume, which can be easily computed for a planar
lattice graph.  If the bipyramid volume is a lower bound for $z_{\G}$, then
the lower bound in Conjecture~\ref{conj:vol-entropy} holds for $\G$.
As we show in many examples, the bipyramid volume is also an excellent
estimate for the spanning tree entropy.  However, in
Section~\ref{sec:counterexample}, we show that it does not always
provide a lower bound.

Let $G$ and $G^*$ be dual $2$-connected graphs embedded on the torus, such
that their faces are topologically disks, and each edge of $G$
intersects its dual edge exactly once and does not intersect any other
edge of $G^*$.  Let $\G$ be the biperiodic graph in $\R^2$, such that
$G=\G/\Lambda$ for $\Lambda$ as above.  These conditions will be
assumed for all graphs and lattice graphs below.

We form the {\em Temperleyan graph} $\gb$ on the torus as follows: The
black vertices of $\gb$ are the vertices of $G$ and of $G^*$; the
white vertices of $\gb$ are the intersections of edges of $G$ and of
$G^*$.  The edges of $\gb$ join a black vertex for each face of $G$
and of $G^*$ to every white vertex incident to the face, so that $\gb$ is a balanced bipartite graph.
Let $\Gb$ be the biperiodic bipartite graph, such that $\gb=\Gb/\Lambda$.
See Figure~\ref{fig:triaxial}.

\begin{definition}\label{def:bipyramid}
Let $B_n$ denote the hyperbolic regular ideal bipyramid whose link
polygons at the two coning vertices are regular $n$--gons.
The hyperbolic volume of $B_n$ is given by
$$ \vol(B_n) = 2n\,L(\pi/n), \quad \text{for} \ L(\theta)=-\int_0^\theta \log|2\sin t|\, dt,$$
where $L(\theta)$ is the Lobachevsky function.
Since $B_3$ consists of two regular ideal tetrahedra and $B_4$ is the hyperbolic regular ideal octahedron, 
$$\vol(B_3)=2\vtet\approx 2.02988 \quad \text{and} \quad \vol(B_4)=\voct\approx 3.66386.$$
We allow $n=2$, but note $\vol(B_2)=0$.
For $n\geq 3,\ \vol(B_n)< 2\pi\log(n/2)$ and grows asymptotically like $2\pi\log(n/2)$ \cite{Adams:bipyramids}.
See the table of values of $\vol(B_n)$ adapted from \cite{Adams:bipyramids}.

For $v\in VG$ and $f\in FG$, let $|v|$ and $|f|$ denote their degree; i.e., the number of incident edges.
Define the \emph{bipyramid volume} of a toroidal graph $G$ as
$$ \volbp(G) = \sum_{f \in FG}\vol(B_{|f|}).$$
\end{definition}

\begin{figure}
\hspace*{-1.2in}
  \begin{minipage}[b]{0.85\textwidth}\centering 
\begin{tabular}{ccc}
  \includegraphics[height=1.3in]{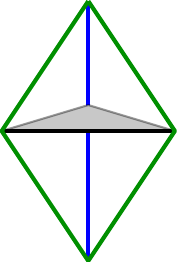} \qquad &
  \includegraphics[height=1.3in]{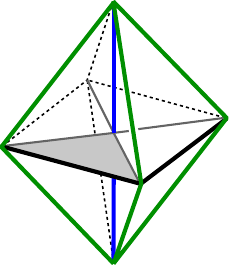} \qquad &
  \includegraphics[height=1.3in]{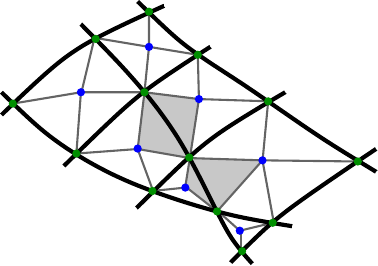}\\
  (a) & (b) & (c) 
\end{tabular}
\caption{(a) A tetrahedron with stellating (blue), vertical (green), and horizontal (black) edges.
  A link triangle for an ideal vertex at $\infty$ is shaded.
  (b) Tetrahedra are glued at the stellating edge, centrally triangulating faces of $G$ on the torus.
  (c) Shaded triangles indicate the four tetrahedra glued along one horizontal edge.\\ \\
  Right: Volumes of hyperbolic regular ideal $n$-bipyramids.}
\label{Fig:quad-graph}
  \end{minipage}
  \hspace{-1cm}
\begin{minipage}[b]{0.1\textwidth}\centering 
\begin{tabular}{|c|c|}
\hline  
$n$ & $\vol(B_n)$ \\ \hline
2 & 0 \\ \hline 
3 & 2.02988 \\ \hline
4 & 3.66386 \\ \hline
5 & 4.98677 \\ \hline
6 & 6.08965 \\ \hline
7 & 7.03257 \\ \hline
8 & 7.85498 \\ \hline
9 & 8.58367 \\ \hline
10 & 9.23755 \\ \hline
11 & 9.83040 \\ \hline
12 & 10.37255 \\ \hline
13 & 10.87192 \\ \hline
14 & 11.33474 \\ \hline
15 & 11.76597 \\ \hline
20 & 13.56682 \\ \hline
100 & 23.67095 \\ \hline
1,000 & 38.13817 \\ \hline
1,000,000 & 81.5409 \\ \hline
\end{tabular}
\vspace*{-0.1in}
\end{minipage}
\end{figure}

\begin{prop}\label{prop:vol}
  $0 \ < \ \vol(G) \ \leq \ \volbp(G)+\volbp(G^*)$.
\end{prop}  
\begin{proof}
This is essentially proved in the last part of
\cite[Theorem~7.5]{ckp:gbal}, and the extra condition in that theorem
is not required for this part.  Our conditions on $G$ ensure that
$(T^2\times I)-L$ is hyperbolic, and that after collapsing bigons if
needed, $(T^2\times I)-L$ admits an ideal, positively oriented
triangulation \cite[Lemma~2.5]{ckp:gbal}.  (In Section~\ref{sec:rak},
we add conditions for this hyperbolic structure to be right-angled.)
Thus, $\vol(G)>0$.

By \cite[Lemma~2.6]{ckp:gbal}, the ideal tetrahedra in this
triangulation can be combined around every stellating edge as in
Figure~\ref{Fig:quad-graph} to form a hyperbolic ideal bipyramid over
each face of $L$.  Since the projection of $L$ is the medial graph of
$G$, these are exactly the faces of $G$ and of $G^*$.
By~\cite{Adams:bipyramids}, the maximal volume of hyperbolic
bipyramids is achieved by the regular bipyramids, whose volumes sum to
$\volbp(G)+\volbp(G^*)$.
\end{proof}  

\begin{definition}\label{def:zfd}
Let $\zfd_G = |VG|z_{\G}$, which is the spanning tree entropy per fundamental domain of $\G$.
The spanning tree entropy per vertex, rather than per fundamental
domain, is usually computed.  To bound $z_{\G}$, we define the
\emph{bipyramid volume of} $\G$ as follows:
$$ \vg = \frac{\volbp(G)+\volbp(G^*)}{|VG|}.$$
All faces of $\gb$ are quads, and $|F\gb|=2|EG|$, so
$\displaystyle \volbp(\gb)= |F\gb|\vol(B_4) = 2|EG|\voct$. Thus,
$$ \vbar(\G)=\frac{|EG|\voct}{|VG|} = \frac{\volbp(\gb)}{2|VG|}.$$
\end{definition}
\begin{prop}\label{prop:eq1}
If $\G$ satisfies the the inequalities in~(\ref{eq:volbp-zfd}), then Conjecture~\ref{conj:vol-entropy} holds for $\G$.
\begin{equation}\label{eq:volbp-zfd}
 \volbp(G)+\volbp(G^*) \leq 2\pi\,\zfd_G \leq |EG|\voct \quad \Longleftrightarrow \quad \vg \leq 2\pi\,z_{\G} \leq \vgb. 
\end{equation}
\end{prop}  
\begin{proof}
  By Proposition~\ref{prop:vol},  \ $0 \ < \ \vol(\G) \ \leq \ \vg$.
\end{proof}  
We use Proposition~\ref{prop:eq1} to prove
Conjecture~\ref{conj:vol-entropy} for infintely many planar lattice graphs
in Section~\ref{sec:infinite}.  However, in
Section~\ref{sec:counterexample} we discuss a planar lattice graph for which
Conjecture~\ref{conj:vol-entropy} holds, but which does not satisfy
the lower bound in~(\ref{eq:volbp-zfd}).

As mentioned above, for the regular triangular, square and hexagonal lattice graphs,
$$ \nu^{\lozenge}(\mathcal G_{\triangle}) = 2\pi\,z_{\G_{\triangle}} =  10\vtet, \ \ \nu^{\lozenge}(\mathcal G_{\square}) = \vbar(\mathcal G_{\square}) = 2\pi\,z_{\G_{\square}} =   2\voct, \ \ \nu^{\lozenge}(\mathcal G_{\hexagon}) = 2\pi\,z_{\G_{\hexagon}} =  5\vtet.$$
Hence, the lower bound in Conjecture~\ref{conj:vol-entropy} is sharp for $\G_{\triangle},\,\G_{\square},\,\G_{\hexagon}$, and the upper bound is sharp for $\G_{\square}$;
see, e.g., \cite[Theorems~12, 13]{ckl:mm_voldet}.

\begin{question}\label{Q:equality}
  Is there a planar lattice graph $\G$, other than $\G_{\triangle},\,\G_{\square},\,\G_{\hexagon}$, for which $\vg=2\pi\,z_{\G}$?
  Is there a planar lattice graph $\G$, other than $\G_{\square}$, for which $\vgb=2\pi\,z_{\G}$?
\end{question}

The Tait graphs for the square weave and the triaxial link as in
Figure~\ref{fig:triaxial} are $\G_{\triangle},\,\G_{\square},\,\G_{\hexagon}$.
Question~\ref{Q:equality} asks whether these are the only two
biperiodic alternating links, such that
$\volbp(G)+\volbp(G^*)=2\pi\,\zfd_G$; and whether the square weave is
the only one such that $2\pi\,\zfd_G=|EG|\voct$ (see \cite{ckl:mm_voldet,
  ckp:gmax}).
Question~\ref{Q:equality} seems related to \cite[Corollary
  20]{Kaplan-Kelly}, which established that the square weave and the
triaxial link are the only semi-regular links with certain geometric
properties, but these properties may not be required for these
equalities to hold.

\begin{table}
\centering  
\begin{tabular}{lc|c|c|c}
\hspace*{0.6cm} Planar lattice graph $\G$ & $|VG|$ & $\vg/2\pi$ & $z_{\G}$ & $\vgb/2\pi$  \\
\hline \rule{0pt}{3ex} 
1. \ Triangular ($3^6$) & 1 & 1.61533 & 1.61533 & 1.74937  \\
\ 2. \ Square ($4^4$) & 1 & 1.16624 & 1.16624 & 1.16624  \\
\ 3. \ Hexagonal ($6^3$) & 2 & 0.80766 & 0.80766 & 0.87468  \\
\ 4. \ Kagome (3-6-3-6) & 3 & 1.12157 & 1.13570 & 1.16624  \\
\ 5. \ Square-octagon (4-8-8) & 4 & 0.78139 & 0.78668 & 0.87468  \\
\ 6. \ Medial(4-8-8) & 6 & 1.10405 & 1.12171 & 1.16624  \\
\ 7. \ 3-12-12  & 6 & 0.70590 & 0.72056 & 0.87468  \\
\ 8. \ 3-4-6-4 & 6 & 1.14390 & 1.14480 & 1.16624 \\
\ 9. \ 4-6-12 & 12 & 0.76795 & 0.77780 & 0.87468 \\ 
10.\, Cairo pentagonal lattice graph & 6 & 0.93886 & 0.94057 & 0.97187 \\
11.\, Lattice graph shown in Figure~\ref{fig:typical}  & 9 & 0.84361 & 0.84744 & 0.90708  \\ 
12.\, Lattice graph shown in Figure~\ref{fig:lattice-list} (\#12) & 3 & 1.07689 & 1.10365 & 1.16624  \\
13.\, Lattice graph shown in Figure~\ref{fig:lattice-list} (\#13) & 2 & 1.39079 & 1.39928 & 1.74937  \\ 
14.\, $3^2$-4-3-4 & 4 & 1.40830 & 1.41086 & 1.45780 \\
15.\, $4^4$; $3^3$-$4^2$ & 3 & 1.32761 & 1.32774 & 1.36062 \\
16.\, $3^6$; $3^3$-$4^2$ & 3 & 1.47731 & 1.47739 & 1.55499 \\ \\

\end{tabular}
\caption{Geometric bounds for $z_{\G}$ for 16 planar
  lattice graphs. See Figures~\ref{fig:lattice-list}, \ref{fig:typical}
  and \ref{fig:dual-list} for pictures of the planar lattice graphs.}
\label{table1}
\end{table}

\begin{center}
  \begin{figure}[h]
    \begin{tabular}{ccccc}
\includegraphics[height=1in]{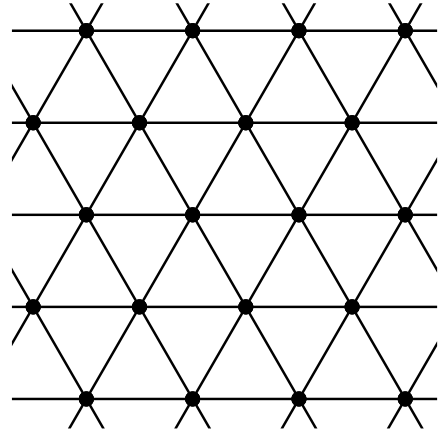} & \includegraphics[height=1in]{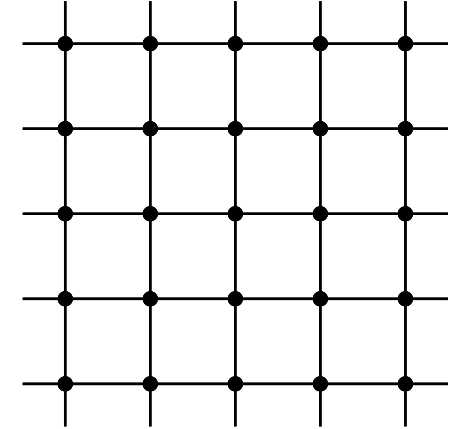} & \includegraphics[height=1in]{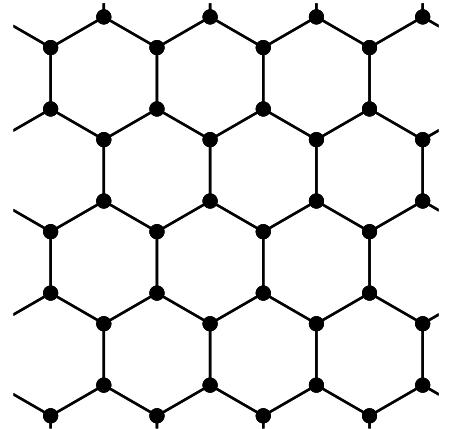} & \includegraphics[height=1in]{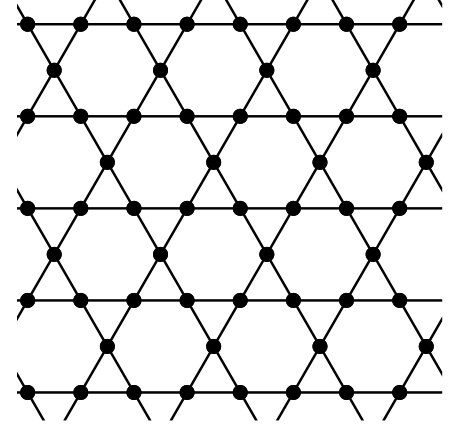} & \includegraphics[height=1in]{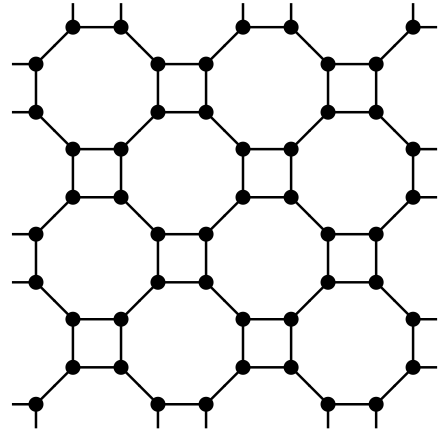} \\
      {\small Lattice graph \#1 } & {\small Lattice graph \#2 } & {\small Lattice graph \#3 } &{\small Lattice graph \#4 } &{\small Lattice graph \#5 }\\
      &&&& \\
\includegraphics[height=1in]{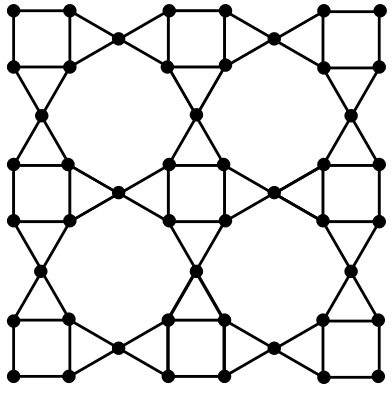} & \includegraphics[height=1in]{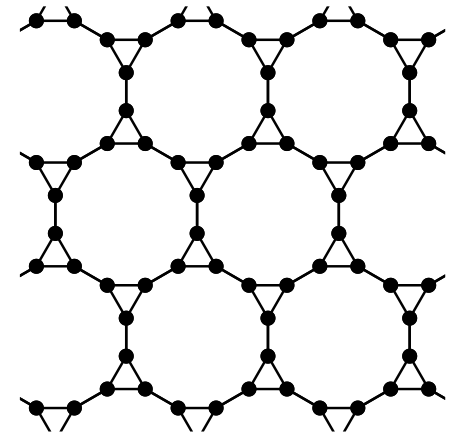} & \includegraphics[height=1 in]{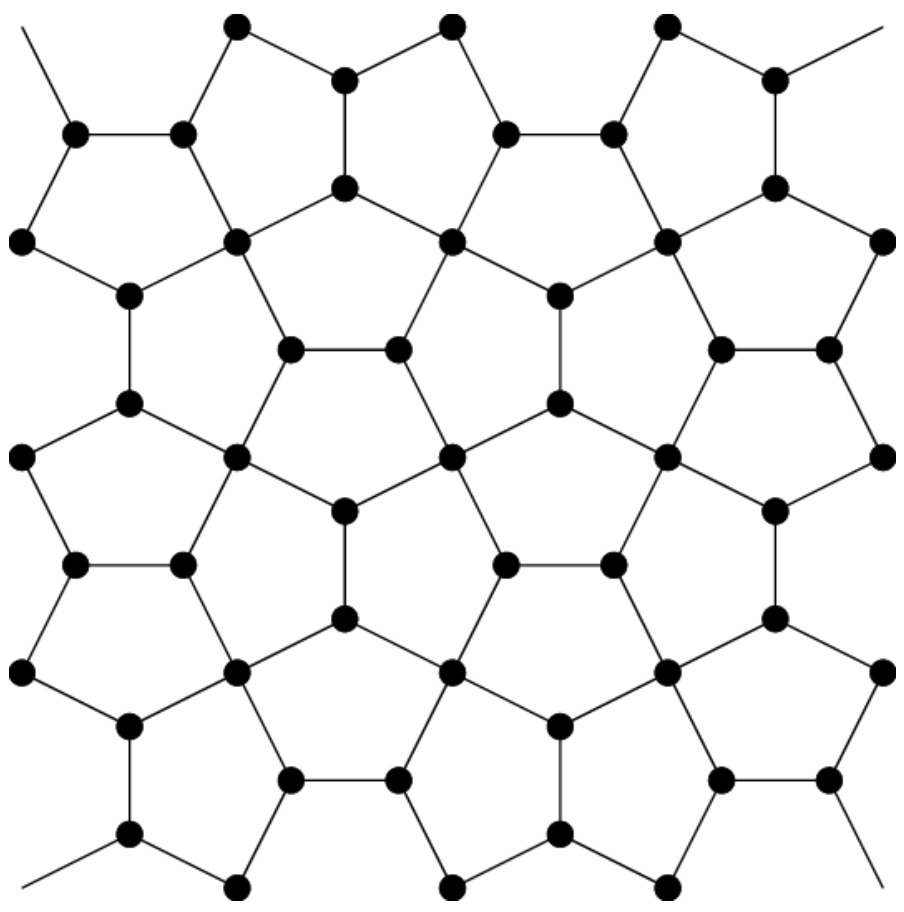} & \includegraphics[height=1in]{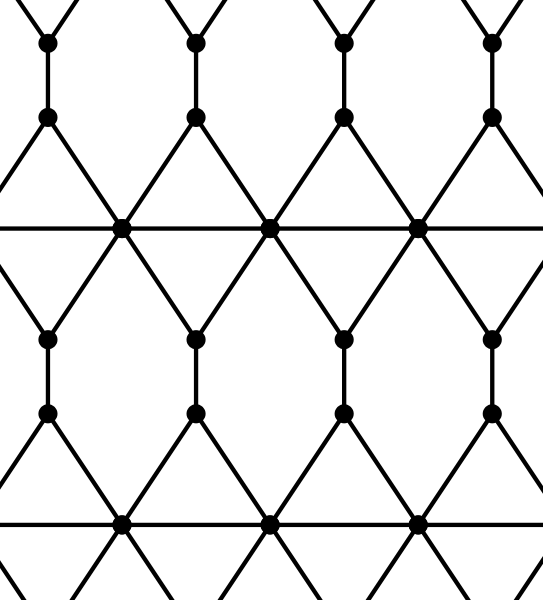} & \includegraphics[height=1in]{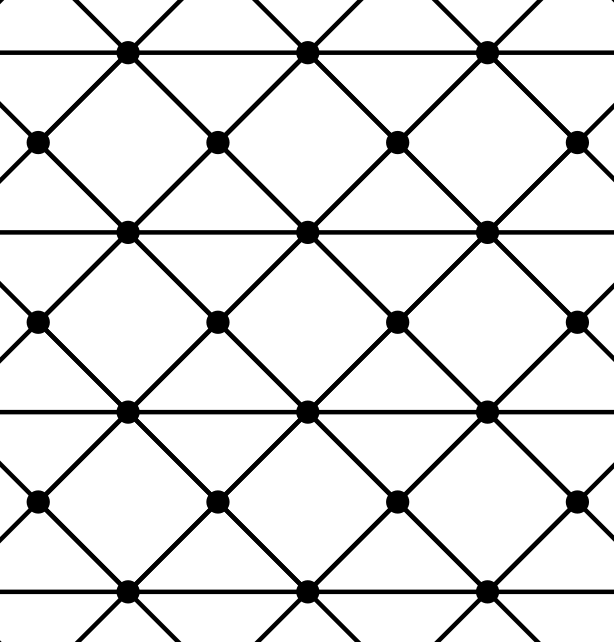}   \\
      {\small Lattice graph \#6 } & {\small Lattice graph \#7 } & {\small Lattice graph \#10 }  & {\small Lattice graph \#12 } & {\small Lattice graph \#13 } \\
      &&&& \\
 & \includegraphics[height=1in]{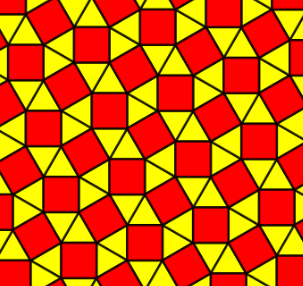} & \includegraphics[height=1in]{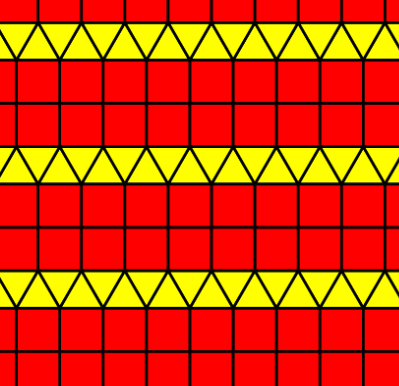} & \includegraphics[height=1in]{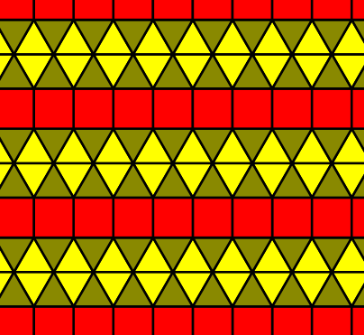} & \\                
       & {\small Lattice graph \#14 } & {\small Lattice graph \#15 } &{\small Lattice graph \#16 } &\\
      \end{tabular}
    \caption{Some planar lattice graphs in Table \ref{table1}.  Figures for Lattice graphs \#1-5 and~7 are from \cite{TeuflWagner}.  Figures for Lattice graphs \#14-16 are from \cite{semi-regular-wiki}.  For Lattice graphs \#8-9, see Figure~\ref{fig:dual-list}.} 
\label{fig:lattice-list}
\end{figure}
\end{center}

\begin{center}
\begin{figure}
  \includegraphics[height=1.5in]{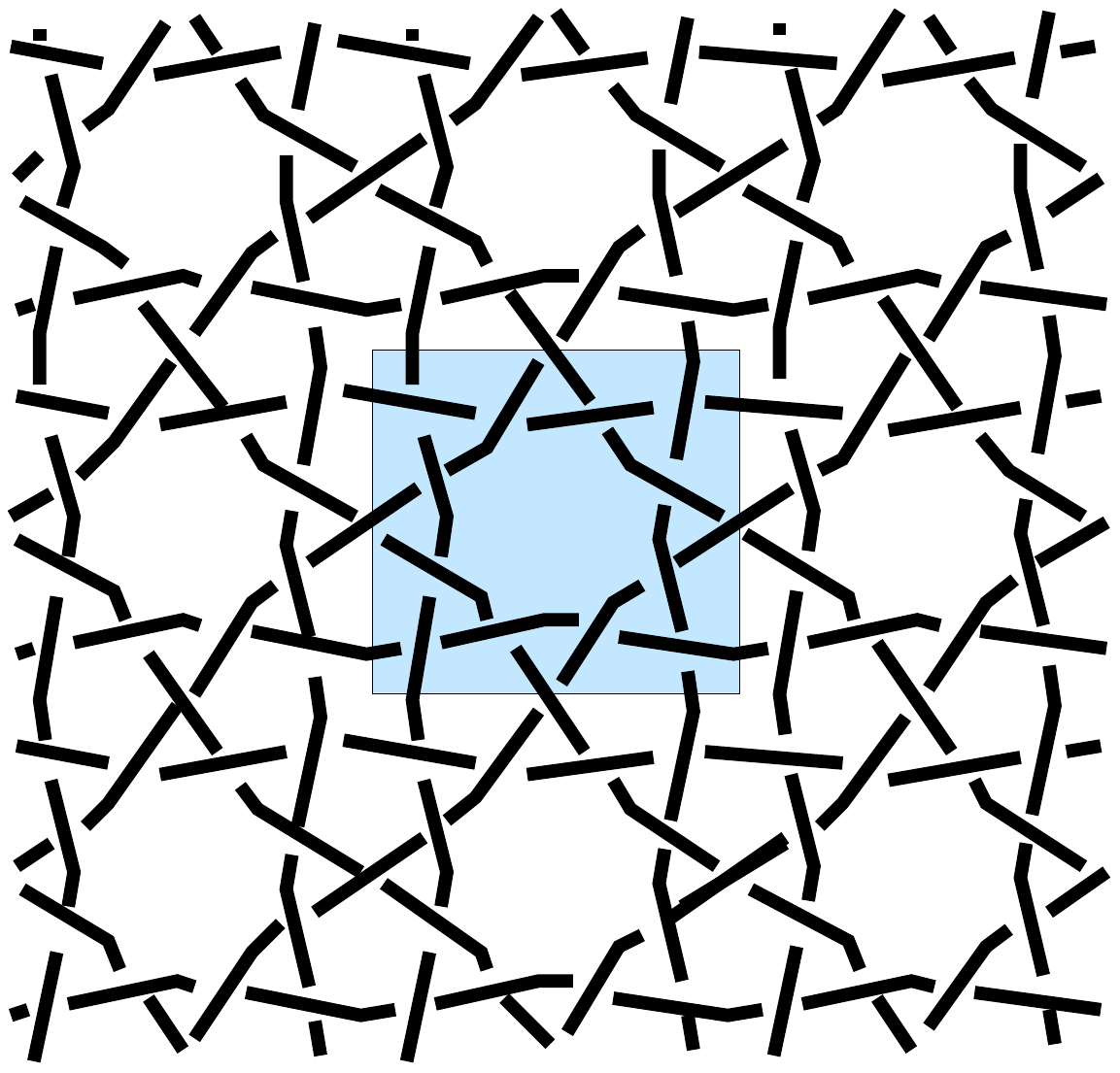} \hspace{0.5in}
  \includegraphics[height=1.5in]{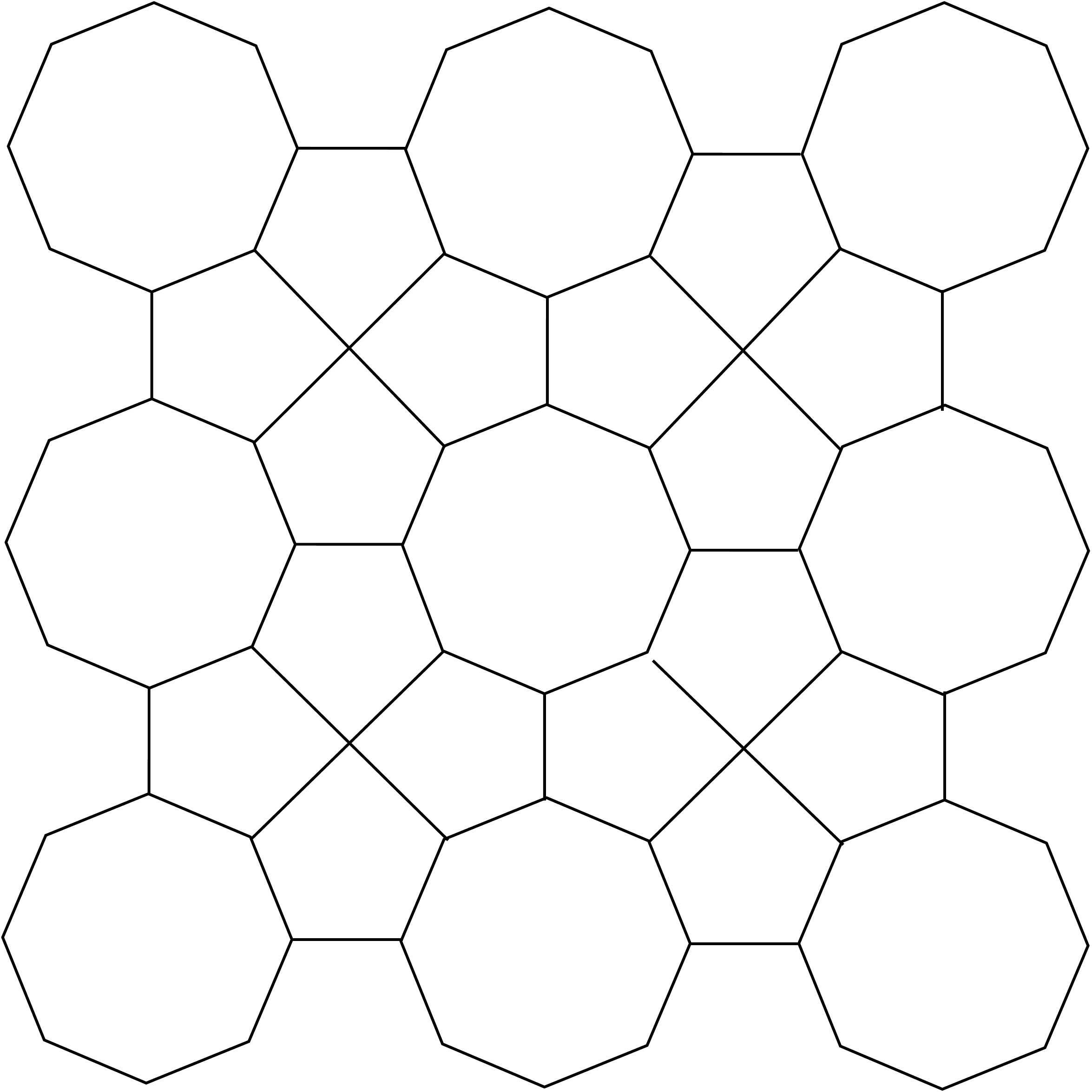} \hspace{0.5in}
    \caption{Biperiodic alternating link whose Tait graph is Lattice graph \#11 in Table~\ref{table1}.}
  \label{fig:typical}
\end{figure}
\end{center}

\subsection{Examples}
In Table~\ref{table1}, we show that $\vg \leq 2\pi\,z_{\G} \leq \vgb$
holds for 16 planar lattice graphs.  Therefore,
Conjecture~\ref{conj:vol-entropy} holds for these lattice graphs, and
infinitely many other lattice graphs using results in
Section~\ref{sec:infinite}.  The values in each column for lattice
graphs \#1-7 and 12-13 are known by exact computations, and equal
values in Table~\ref{table1} indicate exact equality (see
\cite{ckl:mm_voldet}). For the remaining lattice graphs, $z_{\G}$ is
computed numerically.  Below, we discuss each of these examples.

\begin{enumerate}
\item[1-3.] For the regular lattice graphs, the equal values in the table are exactly equal.
\item[4.] The kagome lattice graph is the medial graph of the hexagonal lattice graph, as in Theorem~\ref{thm:medial}.
\item[5.] Exact values for the square-octagon lattice graph are computed in \cite[Theorem~19]{ckl:mm_voldet}.
\item[6.] This lattice graph is the medial graph of the square-octagon lattice graph, as in Theorem~\ref{thm:medial}.
\item[7.] The 3-12-12 lattice graph is the truncation of the hexagonal lattice graph, as in Theorem~\ref{thm:G3}.
\item[8.] The 3-4-6-4 lattice graph is the medial graph of the kagome lattice graph.
  Its numerical $z_{\G}$ value is computed in \cite[Section~3.4]{Chang-Wang-2006}.  
\item[9.] The numerical $z_{\G}$ value for this lattice graph is computed in \cite[Section~3.5]{Chang-Wang-2006}. 
\item[10.] The numerical $z_{\G}$ value for this lattice graph is computed in \cite[Corollary~2.2]{Li-Yan}.
\item[11.] This lattice graph is the Tait graph of the link shown in Figure~\ref{fig:typical}, and  
  its $z_{\G}$ value is computed numerically in \cite[Section~1.4]{ckl:mm_voldet}.
\item[12.] Exact values for this lattice graph are computed in \cite[Theorem~15]{ckl:mm_voldet}.  
\item[13.] Exact values for this lattice graph are computed in \cite[Theorem~17]{ckl:mm_voldet}.
\item[14.] This lattice graph is the dual of Lattice graph \#10.
\item[15-16.] We computed the numerical $z_{\G}$ values using the methods discussed in \cite{ckl:mm_voldet}.
\end{enumerate}

\subsection{Non-example}\label{sec:counterexample}

Let $\G$ be the $3^3$-$4^2$ lattice graph, which is shown below.
Using the methods discussed in \cite{ckl:mm_voldet}, we computed the characteristic polynomial $p(z,w)$ of the toroidal dimer model for $\G$:
$$ p(z,w)=-wz^4+w^2z^2+11wz^3+w^2z-24wz^2+z^3+11wz+z^2-w. $$
The Mahler measure $\m(p(z,w))=\zfd_G$ by \cite{KOS}, as explained in \cite[Proposition~5.3]{ck:det_mp}.

Conjecture~1 of \cite{ckl:mm_voldet} would imply that $\displaystyle \volbp(G)+\volbp(G^*) \leq 2\pi\,\m(p(z,w))$.
However, we numerically computed $\m(p(z,w))$, and $\vol((T^2\times I)-L)$ using Snappy \cite{snappy}:
\begin{align*}
\vol((T^2\times I)-L) &  < & 2\pi\, \m(p(z,w)) &  < & \volbp(G)+\volbp(G^*) &  < & |EG|\voct\\
17.55732 \ &< & 17.67995 \ &< & 17.69718 \ &< & 18.31931\\
\end{align*}
Therefore, $\G$ does not satisfy the lower bound in~(\ref{eq:volbp-zfd}).
The biperiodic alternating link $\LL$ whose Tait graph is $\G$ is a counterexample to \cite[Conjecture~1]{ckl:mm_voldet}.
Nevertheless, $\G$ satisfies Conjecture~\ref{conj:vol-entropy}.

\

In contrast to the examples in Table~\ref{table1}, in this case we have (with $|VG|=2$):
\begin{center}
  \begin{tabular}{lc}
    \includegraphics[height=1in]{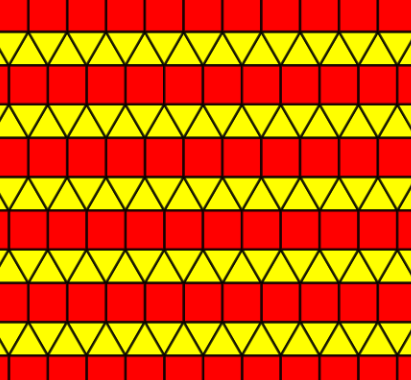}  &
    \qquad
    \begin{tabular}{ccccccc}
   $\ \qquad \quad \vol(\G)$ & $<$ & $2\pi\,z_{\G}$ & $<$ & $\vg$ & $<$ & \!\!\!\!\!\! $\vgb$ \\ \\
   $2\pi*\big(\quad 1.39717$  & $<$ & $1.40693$ & $<$ & $1.40830$  & $<$ & $1.45780 \ \quad \big)$ \\ \\ \\ \\ \\ \\
\end{tabular}
\end{tabular}
\end{center}

\vspace*{-0.6in}

\noindent
The numerical $z_{\G}$ value is also computed in \cite[Section~3.1]{Chang-Wang-2006}.
We do not know any other counterexamples to
\cite[Conjecture~1]{ckl:mm_voldet}. The above figure from
\cite{semi-regular-wiki} suggests that $\G$ is similar to Lattice graphs
\#14-16 in Table~\ref{table1}, but for these lattice graphs $\vg < 2\pi\,z_{\G}$,
although their values of $\vg$ and $2\pi\,z_{\G}$ are very close!  Adding
parallel edges to any of these four lattice graphs, as in
Theorem~\ref{thm:Gs} below, also results in lattice graphs $\G$ with
$\vg < 2\pi\,z_{\G}$.

\begin{question}\label{Q:counterexample}
  For which other planar lattice graphs is $\displaystyle 2\pi\,z_{\G} < \vg$?
\end{question}

\section{Finite planar graphs}\label{sec:planar}

For a finite connected graph $\Gamma$, the number of spanning
trees $\tau(\Gamma)$ is an important measure of its complexity.
For finite planar graphs, many interesting bounds are known (see,
e.g., \cite{Buchin-Schulz}), and asymptotic enumeration of spanning
trees by finite planar graphs that approach a planar lattice graph is
sometimes exactly computable (see, e.g., \cite{ckl:mm_voldet, KOS,
  Lyons}).

Any finite connected planar graph $\Gamma$ is the Tait graph of an
alternating link $K$ in $S^3$, so the link diagram of $K$ projects to
the medial graph of $\Gamma$.  We define
$$ \vol(\Gamma)=\vol(S^3-K),$$
$$\volbp(\Gamma)= \sum\nolimits_{f \in \{\substack{\text{bounded}\\ \text{faces of}} \ \Gamma\}}\vol(B_{|f|}) \quad \text{and} \quad
\nu^{\lozenge}(\Gamma)= \frac{\volbp(\Gamma)+\volbp(\Gamma^*)}{|V\Gamma|}.$$

\begin{prop}\label{prop:alt-hyp}
If neither $\Gamma$ nor $\Gamma^*$ has loops or a cut-vertex, and
neither is a cycle graph, then
$$ 0 \ < \ \vol(\Gamma) \ < \ \volbp(\Gamma) + \volbp(\Gamma^*).$$
\end{prop}
\begin{proof}
The alternating diagram is reduced if neither $\Gamma$ nor $\Gamma^*$
has loops.  The alternating link $K$ is prime if neither $\Gamma$ nor
$\Gamma^*$ has a cut-vertex. By the classification of hyperbolic
alternating links due to Thurston and Menasco, $K$ is hyperbolic with
finite volume if it is reduced, prime and not a $(2,n)$-torus link.
With our conditions, $K$ is hyperbolic with finite volume if neither
$\Gamma$ nor $\Gamma^*$ is a cycle graph.  Thus, $\vol(\Gamma)>0$.
Since the projection of $K$ is the medial graph of $\Gamma$, the faces
of $K$ are exactly the faces of $\Gamma$ and of $\Gamma^*$.  Hence,
our conditions on $\Gamma$ satisfy those in~\cite[Theorem
  4.1]{Adams:bipyramids}, which implies the strict upper bound.
\end{proof}

We conjecture the following geometric bounds for the number of spanning trees of a finite connected planar graph:

\begin{conjecture}\label{conj:planar}
If $\vol(\Gamma)>0$ then
  $$ \vol(\Gamma) < 2\pi\log\tau(\Gamma) < |E\Gamma|\voct.$$
\end{conjecture}

Conjecture~\ref{conj:planar} is a restatement of \cite[Conjectures~1.1
  and 1.10]{ckp:gmax}.  The conjectured upper bound $|EG|\voct$ is
essentially due to Kenyon; see \cite[Conjecture~2.3]{ckp:gmax} for
more details.

The lower bound, called the Vol-Det Conjecture, has been verified for
all hyperbolic alternating knots with up to 16 crossings, i.e.,
$|E\Gamma|\leq 16$, as well as all $2$-bridge links and alternating
$3$-braids in \cite{Burton}.  Moreover, the constant $2\pi$ cannot be
improved in the lower bound: By \cite[Corollary~1.11]{ckp:gmax}, if
$\alpha<2\pi$, then there exist hyperbolic alternating knots $K$, such
that $\alpha\log\tau(\Gamma) < \vol(\Gamma)$.

\begin{theorem}
  Let $\G$ be a planar lattice graph that satisfies
  $$ \volbp(G)+\volbp(G^*) \ < \ 2\pi\,\zfd_G \ < \ |EG|\voct.$$ 
  Let $\Gamma_n$ be a sequence of connected planar graphs with bounded average degree that converges to $\G$ as in \cite[Theorem~3.2]{Lyons}.
  Then for all but finitely many $n$, 
  $$ \vol(\Gamma_n)  \ < \  2\pi\log\tau(\Gamma_n)  \ < \  |E\Gamma_n|\voct.$$
\end{theorem}  
\begin{proof}
By \cite[Theorem~8]{ckl:mm_voldet}, $\nu^{\lozenge}(\Gamma)$
behaves well under the type of convergence as in \cite[Theorem~3.2]{Lyons}; namely,
$$ \Gamma_n\to \G \quad \Rightarrow \quad \lim_{n\to\infty} \nu^{\lozenge}(\Gamma_n) = \vg,$$
Thus, the lower bound is a restatement of \cite[Theorem~3]{ckl:mm_voldet}.
For the upper bound, the convergence is similar: $|E\Gamma_n|/|V\Gamma_n|\to |EG|/|VG|$.
\end{proof}

Thus, whenever the inequalities in~(\ref{eq:volbp-zfd}) are strict
inequalities for a planar lattice graph, we obtain infinitely many planar graphs that satisfy
Conjecture~\ref{conj:planar}.

\section{Infinitely many proven cases for Conjecture~\ref{conj:vol-entropy}}\label{sec:infinite}

In this section, we use three different ways to construct infinite
families of planar lattice graphs that are proven cases for
Conjecture~\ref{conj:vol-entropy} using bipyramid volume in an
essential way.  Namely, we prove that these planar lattice graphs satisfy
the inequalities in~(\ref{eq:volbp-zfd}), so then Proposition~\ref{prop:eq1}
implies Conjecture~\ref{conj:vol-entropy} for $\G$.

\subsection{Parallel edges}

For any planar lattice graph $\G$, let $\G_s$ denote the planar lattice graph for which
every edge of $\G$ is replaced by $s$ parallel edges.  Then (see \cite{TeuflWagner})
$$ z_{\G_s} = z_{\G} + \log s. $$ 
In Theorem~\ref{thm:Gs}, we show that if the inequalities in~(\ref{eq:volbp-zfd}) hold
for $\G$, then they hold with strict inequalities for $\G_s$ for any $s\geq 2$.

\begin{lemma}\label{lemma:Bns}
  Fix an integer $s\geq 2$. Then for all $n\geq 2$,
  $$ \vol(B_{ns}) - \vol(B_n) < 2\pi\log s. $$
\end{lemma}
\begin{proof}
We can modify the proof of \cite[Theorem 2.2]{Adams:bipyramids}.
Decompose $B_n$ into $n$ isometric ideal tetrahedra $T_n$, such that
each $T_n$ intersects the stellating edge of $B_n$ with an edge of dihedral
angle $2\pi/n$, as in Figure~\ref{Fig:quad-graph}(a,b).  Thus, $\vol(B_n)=n\cdot\vol(T_n)$.
$$ \vol(B_{ns}) < \vol(B_n) + 2\pi\log s \quad \Longleftrightarrow \quad \vol(T_{ns}) < \frac{1}{s}\vol(T_n) + \frac{1}{ns}2\pi\log s.$$

Let $f(n)=\vol(T_{ns})$ and $g(n)= \frac{1}{s}\vol(T_n) + \frac{1}{ns}2\pi\log s$.
We must prove $f(n)<g(n)$ for all $n\geq 2$.
By \cite[Theorem 2.2]{Adams:bipyramids}, $f(2)<g(2)$ and the functions $f$ and $g$ agree asymptotically as $n\to\infty$.
If $f(a)>g(a)$ for some $a>2$, then it must be true that $f'(b)<g'(b)$ for some $b>a$.
But we now prove that $f'(x)>g'(x)$ for all $x\geq 2$.

$$ f'(n)=\frac{2\pi}{n^2s}\log\left(2\sin\left(\frac{\pi}{ns}\right)\right)$$
$$ g'(n)=\frac{2\pi}{n^2s}\log\left(2\sin\left(\frac{\pi}{n}\right)\right) - \frac{2\pi}{n^2s}\log s =
\frac{2\pi}{n^2s}\log\left(\frac{2}{s}\sin\left(\frac{\pi}{n}\right)\right).$$
Since $\displaystyle h(\theta)=\frac{\sin \theta}{\theta}$ is a decreasing function for $0<\theta<\pi$, then
$\displaystyle \frac{\sin(\pi/ns)}{\pi/ns} > \frac{\sin(\pi/n)}{\pi/n}$ for any $s\geq 2$. 
Thus, 
$$ \sin\left(\frac{\pi}{ns}\right) > \frac{1}{s}\sin\left(\frac{\pi}{n}\right).$$
Therefore,
$f'(x)>g'(x)$ for all $x\geq 2$.  This implies that for all $n\geq 2$, $f(n)<g(n)$.
\end{proof}

\ 

\begin{theorem}\label{thm:Gs}
  Let $s\geq 2$. If $\displaystyle \vg \leq 2\pi\,z_{\G}\leq \vgb$, then $\displaystyle \vgs < 2\pi\,z_{\G_s} < \vgbs.$
\end{theorem}
\begin{proof}
We first prove the lower bound.
Every face $f\in FG$ with $|f|>2$ has the same degree in $G_s$, and $\vol(B_2)=0$, so
$$ \volbp(G)+\volbp(G^*) = \sum\nolimits_{f \in FG}\vol(B_{|f|}) + \sum\nolimits_{v \in VG}\vol(B_{|v|}), $$
$$ \volbp(G_s)+\volbp(G_s^*) = \sum\nolimits_{f \in FG}\vol(B_{|f|}) + \sum\nolimits_{v \in VG}\vol(B_{s|v|}). $$
Then by Lemma~\ref{lemma:Bns},
$$ \volbp(G_s^*) - \volbp(G^*) \leq |VG|\max_{v \in VG}\left(\vol(B_{s|v|})-\vol(B_{|v|})\right) < 2\pi |VG|\log s.$$
Dividing by $|VG|$, this implies $\vgs < \vg + 2\pi\,\log s.$
Since we assumed $\vg \leq 2\pi\,z_{\G},$ then
$\displaystyle \vgs < 2\pi\,z_{\G} + 2\pi\,\log s = 2\pi\,z_{\G_s}.$

We now prove the upper bound, assuming $2\pi\,z_{\G}\leq \vgb$. Since $|EG_s|=s|EG|$, we have
\begin{eqnarray*}
\zfd_{G_s} = \zfd_G + \log s & \leq & \frac{|EG|\voct}{2\pi} + \log s \\ 
& = & \frac{1}{s}\left( \frac{s|EG|\voct}{2\pi} + s\log s \right) \\
& < & \frac{1}{s}\left( \frac{|EG_s|\voct}{2\pi} + \frac{(s-1)|EG_s|\voct}{2\pi} \right) \\
& = & \frac{|EG_s|\voct}{2\pi} \\
2\pi\,z_{\G_s} & < & \vgbs.
\end{eqnarray*}
The strict inequality follows because for any $s\geq 2$, 
$\displaystyle \frac{2\pi\log s}{(s-1)\voct} < 2 < |EG|.$
\end{proof}

\subsection{Truncating a $3$-regular planar lattice graph}

For any $3$-regular planar lattice graph $\G$, let $\G'$ denote the lattice graph for which
every vertex is replaced by a triangle (i.e., complete graph $K_3$), and each edge of $\G$ is preserved in $\G'$.
For example, the 3-12-12 lattice graph is obtained by truncating the regular hexagonal lattice graph:\\ \\
\centerline{\includegraphics[height=0.8in]{figures/lattice-3-hexagonal}
\begin{tabular}{c}
  $\longrightarrow$\\ \vspace*{1cm} 
\end{tabular}  
\includegraphics[height=0.8in]{figures/lattice-7-31212}}\\

\vspace*{-0.8cm}

\noindent
Since $\G'$ is also $3$-regular, this process can be repeated indefinitely.
In Theorem~\ref{thm:G3}, we show that if the inequalities in~(\ref{eq:volbp-zfd}) hold
for $\G$, then they hold as strict inequalities for $\G'$.

\begin{theorem}\label{thm:G3}
  Let $\G$ be a $3$-regular planar lattice graph.  If $\displaystyle \vg \leq 2\pi\,z_{\G}\leq \vgb$, then\\
  $\displaystyle \nu^{\lozenge}(\G') < 2\pi\,z_{\G'} < \vbar({\G'})$.
\end{theorem}
\begin{proof}
By definition of $\G'$, $|VG'|=3|VG|$.
By \cite[Theorem 4]{TeuflWagner} with $r=3, s=1$, 
$$ z_{\G'} = \frac{1}{3}z_{\G} + \frac{1}{6}\log 15. $$ 
Recall $\zfd_G = |VG|z_{\G}$.
If we multiply this equation by $3|VG|$,
$$ \zfd_{G'} = \zfd_G + \frac{|VG|}{2}\log 15.$$

We first prove the lower bound.
Let $F_n$ and $F_n'$ denote the number of $n$--faces of $G$ and $G'$, respectively.
By definition of $\G'$, $F_{2n}' = F_n$ and $F_3'=|VG|$.
Since $G'$ is a $3$-regular graph on a torus, $F'-E'+V'=0$ and $3V'=2E'$, hence $F'=E'-V'=V'/2=3|VG|/2$.  Thus,
$$ \#\{f\in FG' \ \big| \ |f|>3 \} = \frac{1}{2}|VG|. $$
Let $\Delta = \big(\volbp(G')+\volbp(G'^*)\big) - \big(\volbp(G)+\volbp(G^*)\big)$.

\begin{align*}
  \volbp(G')+\volbp(G'^*) = &  \sum\nolimits_{\underset{|f|>3}{f \in FG'}}\vol(B_{|f|}) & + & \sum\nolimits_{\underset{|f|=3}{f \in FG'}}\vol(B_3) & +
       \sum\nolimits_{v \in VG'}\vol(B_3) \\
  \volbp(G)+\volbp(G^*) =    &  \sum\nolimits_{f \in FG}\vol(B_{|f|}) & & & + \sum\nolimits_{v \in VG}\vol(B_3) \\
       \hline \\ 
\Delta  < &  \ \ \frac{1}{2}|VG|\cdot 2\pi\log 2 & + & |VG|\cdot 2\vtet & + 2|VG|\cdot 2\vtet \hspace*{0.75cm} \\
      = & \ \ |VG| (\pi\log 2 + 6\vtet).
\end{align*}
The first inequality follows by Lemma~\ref{lemma:Bns} with $s=2$ because $F_{2n}' = F_n$.
Thus, since we assumed $\vg \leq 2\pi\,z_{\G}$,
\begin{eqnarray*}
  \volbp(G')+\volbp(G'^*) & < & \volbp(G)+\volbp(G^*) + |VG| (\pi\log 2 + 6\vtet) \\
  & \leq & 2\pi \zfd_G + |VG| (\pi\log 2 + 6\vtet) \\
  & < & 2\pi \left(\zfd_G + \frac{|VG|}{2}\log 15 \right)\\
  & = & 2\pi \zfd_{G'}\\
  \nu^{\lozenge}(\G') & < & 2\pi\,z_{\G'}.
\end{eqnarray*}
The strict inequality follows because $(\pi\log 2+6\vtet) < \pi\log 15$.

We now prove the upper bound, assuming $2\pi\,z_{\G}\leq \vgb$.
Since $|EG'|=3|EG|$, we have
\begin{eqnarray*}
  \zfd_{G'} &=& \zfd_{G} + \frac{|VG|}{2}\log 15 \\
  &\leq & \frac{|EG|\voct}{2\pi} + \frac{|EG|}{3}\log 15 \\
  &=& |EG|\left(\frac{\voct}{2\pi} + \frac{\log 15}{3}\right) \\
  &<& 3|EG|\voct/2\pi \\
  &=& |EG'|\voct/2\pi \\
  2\pi\,z_{\G'} &<& \vbar({\G'}).
\end{eqnarray*}
The strict inequality follows because $\log 15 < 6\voct/2\pi$.
\end{proof}

\subsection{Medial graph of a $3$-regular planar lattice graph}

For any $3$-regular planar lattice graph $\G$, let $\G'$ be the medial graph of
$\G$, which in this case is obtained by subdividing each edge once and
then performing a $Y-\Delta$ transformation at every $3$-valent vertex.
The medial graph $\G'$ is $4$-regular.
For example, the Kagome lattice graph is the medial graph of the regular hexagonal lattice graph:\\ \\
\centerline{\includegraphics[height=0.8in]{figures/lattice-3-hexagonal}
\begin{tabular}{c}
  $\longrightarrow$\\ \vspace*{1cm} 
\end{tabular}  
\includegraphics[height=0.8in]{figures/lattice-4-triaxial}}\\

\vspace*{-0.8cm}

In Theorem~\ref{thm:medial}, we show that if the lower bound
in~(\ref{eq:volbp-zfd}) holds for $\G$, then it holds with a strict
inequality for $\G'$.  In Corollary~\ref{cor:medial}, we then apply
Theorems~\ref{thm:G3} and \ref{thm:medial} to obtain an infinite
family of $4$-regular planar lattice graphs for which both bounds
in~(\ref{eq:volbp-zfd}) hold with a strict inequality.

\begin{theorem}\label{thm:medial}
  Let $\G$ be a $3$-regular planar lattice graph.  If $\displaystyle \vg \leq 2\pi\,z_{\G}$,
  then \mbox{$\displaystyle \nu^{\lozenge}(\G') < 2\pi\,z_{\G'}$.}
\end{theorem}
\begin{proof}
By definition of $\G'$, $\displaystyle |VG'|=|EG|=\frac{3}{2}|VG|$.
By \cite[Theorem 4]{TeuflWagner} with $r=3, s=0$, 
$$ z_{\G'} = \frac{2}{3}z_{\G} + \frac{1}{3}\log 6. $$ 
We multiply this equation by $\displaystyle \frac{3}{2}|VG|$ to obtain
$$ \zfd_{G'} = \zfd_G + \frac{|VG|}{2}\log 6.$$

Using the same notation as in the proof of Theorem~\ref{thm:G3}, by
definition of $\G'$, $F_{3}' = |VG| + F_3$ and $F_n'=F_n$ for all
$n\neq 3$.
Let $\Delta = \big(\volbp(G')+\volbp(G'^*)\big) - \big(\volbp(G)+\volbp(G^*)\big)$.

\begin{align*}
  \volbp(G')+\volbp(G'^*) = &  \sum\nolimits_{\underset{|f|\neq 3}{f \in FG'}}\vol(B_{|f|})  +  \sum\nolimits_{\underset{|f|=3}{f \in FG'}}\vol(B_3) & + &
       \sum\nolimits_{v \in VG'}\vol(B_4) \\
  \volbp(G)+\volbp(G^*) =  &  \sum\nolimits_{f \in FG}\vol(B_{|f|})  +  \sum\nolimits_{v \in VG}\vol(B_3) & \\
       \hline \\ 
\Delta   = &  \quad \big( (|VG| + F_3) - F_3 - |VG| \big)\vol(B_3) & + & |VG'|\vol(B_4) \hspace*{0.75cm} \\
         = &  \quad \frac{3}{2}|VG|\voct. 
\end{align*}

Thus, since we assumed $\vg \leq 2\pi\,z_{\G}$,
\begin{eqnarray*}
  \volbp(G')+\volbp(G'^*) & = & \volbp(G)+\volbp(G^*) + \frac{3}{2}|VG|\voct \\ 
  & \leq & 2\pi \zfd_G + \frac{3}{2}|VG|\voct \\ 
  & = & 2\pi \left(\zfd_G + |VG|\frac{3\voct}{4\pi} \right)\\
  & < & 2\pi \left(\zfd_G + \frac{|VG|}{2}\log 6 \right)\\ 
  & = & 2\pi \zfd_{G'}\\
  \nu^{\lozenge}(\G') & < & 2\pi\,z_{\G'}.
\end{eqnarray*}
The strict inequality follows because $3\voct < 2\pi\log 6$.
\end{proof}

\begin{corollary}\label{cor:medial}
  Let $\G_n$ denote the sequence of $3$-regular planar lattice graphs obtained from the regular hexagonal lattice graph $\G_0=\G_{\hexagon}$
  by truncation as in Theorem~\ref{thm:G3}.  Let $\G'_n$ be the $4$-regular medial graph of $\G_n$ as in Theorem~\ref{thm:medial}.  Then
  $\displaystyle \nu^{\lozenge}(\G'_n) < 2\pi\,z_{\G'_n} < \vbar({\G'_n})$ for all $n\geq 0$.
\end{corollary}
\begin{proof}
The lower bound follows by Theorem~\ref{thm:medial}.  For the upper bound, 
$$|VG_n|=2\cdot 3^n,\quad |VG'_n|=|EG_n|=3^{n+1},\quad |EG'_n|=2\cdot 3^{n+1}.$$
Applying the formulas in the proofs of Theorems~\ref{thm:G3} and \ref{thm:medial},
\begin{eqnarray*}
  \zfd_{G'_n} &=& \zfd_{G_{\hexagon}} + \left(\sum\nolimits_{i=0}^{n-1}|VG_i|/2\right)\log 15 + \frac{|VG_n|}{2}\log 6 \\
  &=& \frac{10\vtet}{2\pi} + \left(\sum\nolimits_{i=0}^{n-1}3^i\right)\log 15 + 3^n\log 6 \\ 
  &=& \frac{10\vtet}{2\pi} + \frac{3^n-1}{2}\log 15 + 3^n\log 6 \\
  &<& 2\cdot3^{n+1}\voct/2\pi \\
  &=& |EG'_n|\voct/2\pi\\
  2\pi\,z_{\G'_n} &<& \vbar({\G'_n}).
\end{eqnarray*}
The strict inequality follows because
$\displaystyle \left(\frac{10\vtet}{2\pi}-\log \sqrt{15}\right) < 3^n\left(\frac{6\voct}{2\pi}-\log 6\sqrt{15}\right)$ for all $n\geq 0$.
\end{proof}

\section{Right-angled volume of a planar lattice graph}\label{sec:rak}

In this section, for dual simple $3$--connected planar lattice graphs
$\G$ and $\G^*$, we define another geometric invariant, which we call
the right-angled volume $\volp(\G)$.  We prove that
Conjecture~\ref{conj:vol-entropy} implies that $\volp(\G) \leq
2\pi\,z_{\G}$.  We then prove the lower bound in
Conjecture~\ref{conj:vol-entropy} for the regular planar lattice
graphs using the isoradial dimer model.

An \emph{orthogonal circle pattern} is a pair of circle packings whose
contact graphs are planar dual graphs, such that dual edges are
perpendicular and no other edges intersect; i.e., the tangent points
of the dual circle packings coincide, and the two tangent lines at
each point are perpendicular. The Koebe-Andreev-Thurston Circle
Packing Theorem has been generalized to orthogonal circle patterns
(see \cite{Felsner-Rote} and references therein): Every $3$--connected
plane graph admits an orthogonal circle pattern representation, which
is unique up to M\"obius transformations.

An \emph{orthogonally dual embedding} of $\G$ and $\G^*$ is a planar
embedding such that the inscribed circles in their faces form an
orthogonal circle pattern.  We will say that $\G$ and $\G^*$ are
\emph{orthogonally dual lattice graphs}, or that $G$ and $G^*$ are
\emph{orthogonally dual graphs} on the torus, if they admit an
orthogonally dual embedding.  Equivalently, the associated Temperleyan
graph $\gb$ is a quadrangulation embedded such that its faces are
right kites, which are kites with at least two right angles, given by
radii of the orthogonal circle pattern.  Many well-known planar
lattice graphs satisfy orthogonal duality; see
Figure~\ref{fig:dual-list}.

When lifted to the universal cover, considered as the plane at
infinity for $\H^3$, this orthogonal circle pattern defines a
right-angled biperiodic ideal hyperbolic polyhedron $\P$ in $\H^3$
whose vertices are the white vertices of $\Gb$.  Let
$P(G)=\P/\Lambda$, which is a hyperbolic polyhedron with finite
volume.  See Figure~\ref{fig:pyr-kite}.

\begin{definition}\label{def:rak}
If $G$ and $G^*$ are orthogonally dual graphs, define the \emph{right-angled volume} of $G$ and of $\G$ as
$$ \volp(G) = 2 \vol(P(G)) \quad \text{and} \quad \volp(\G)=\frac{\volp(G)}{|VG|}.$$
Note that $\volp(G)=\volp(G^*)$.
\end{definition}

\begin{center}
  \begin{figure}
    \begin{tabular}{ccc}
      \hspace*{-0.5in}\includegraphics[width=2in]{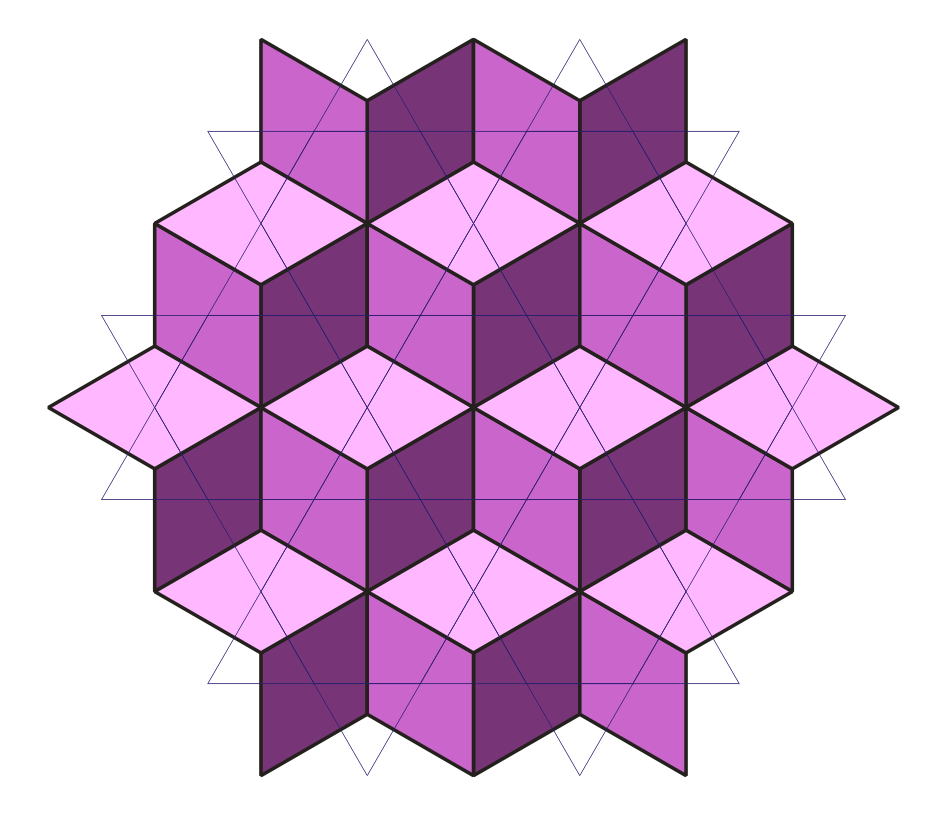} & \includegraphics[width=2in]{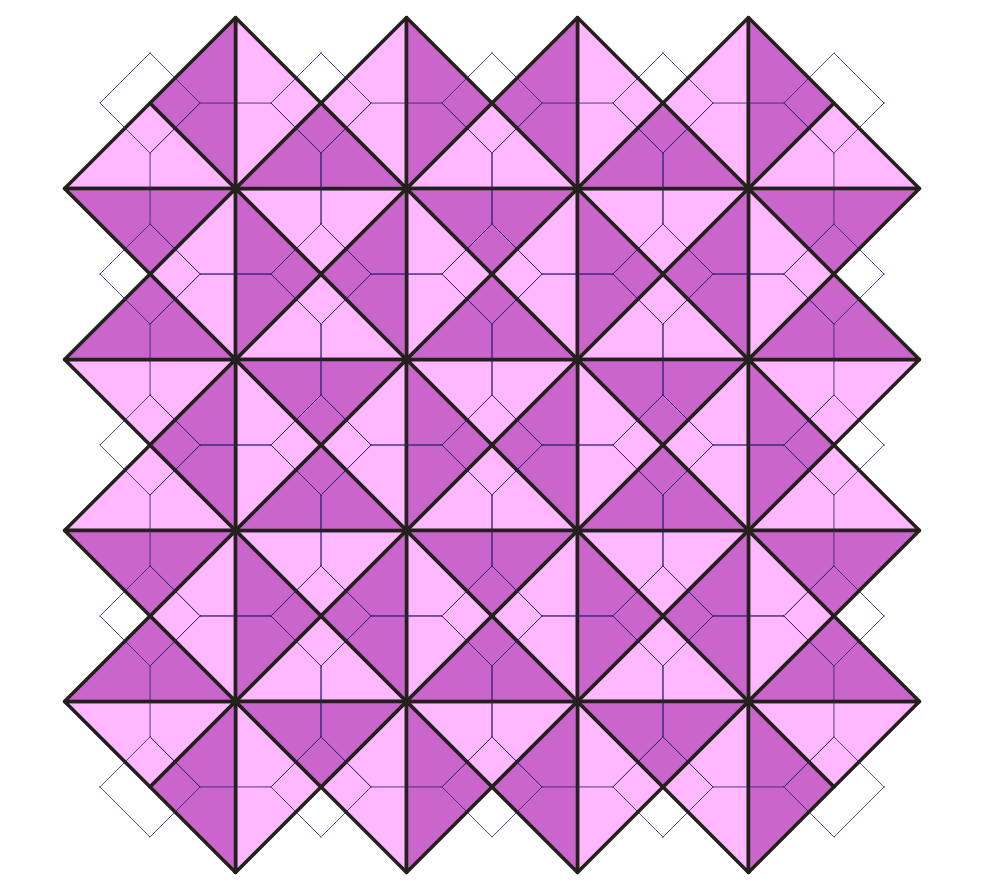} & \includegraphics[width=2in]{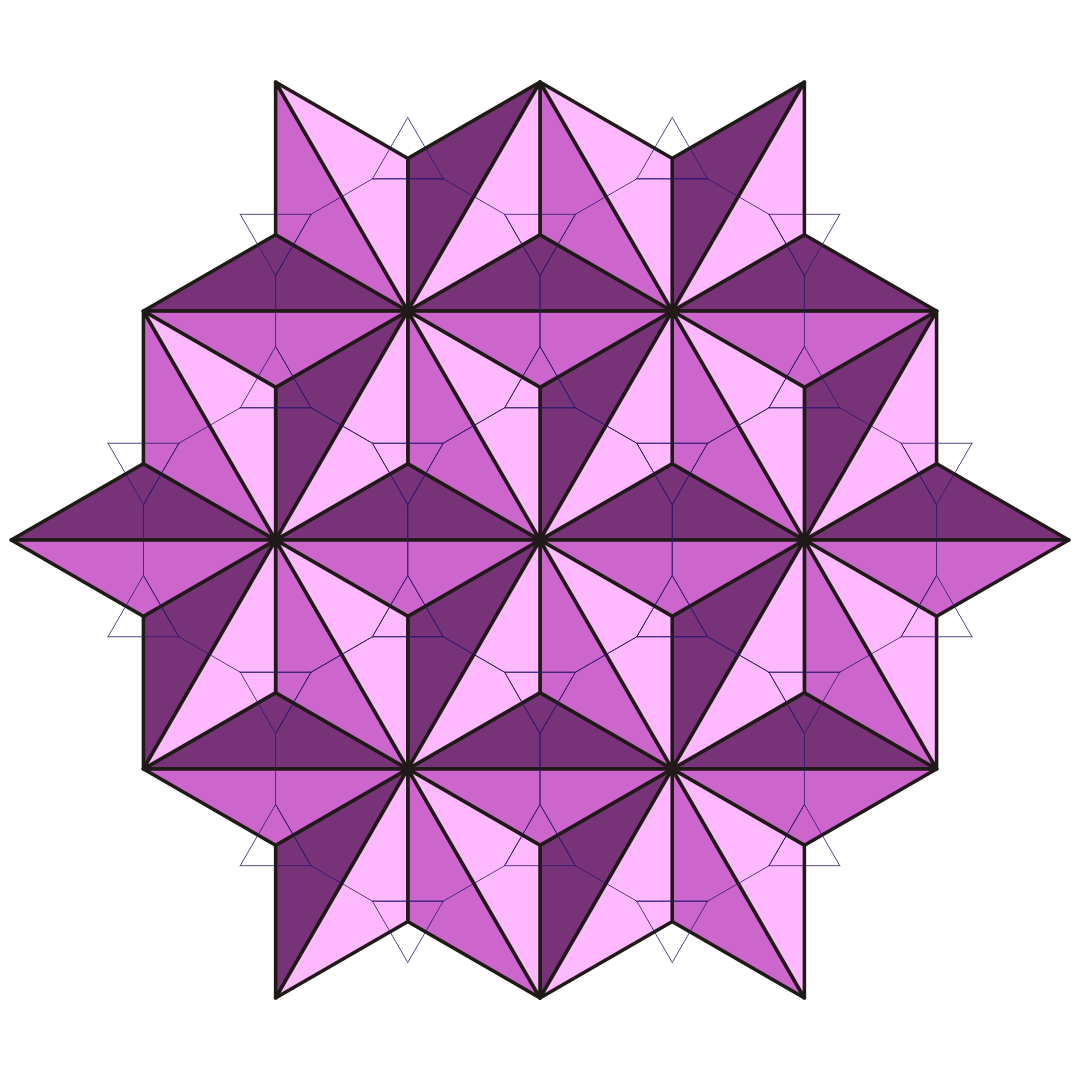} \\
       \hspace*{-0.5in}{\small Lattice graph \#4 } & {\small Lattice graph \#5 } & {\small Lattice graph \#7 } \\ \\
      \hspace*{-0.5in}\includegraphics[width=2.1in]{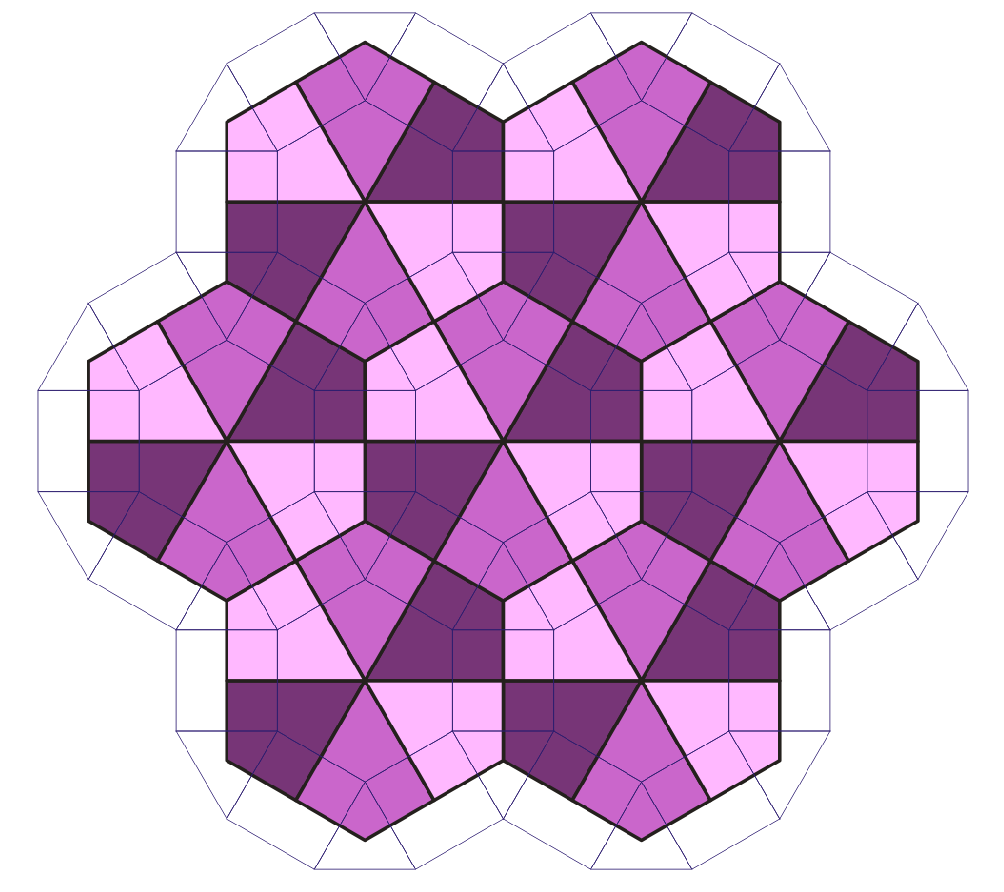} & \includegraphics[width=2.2in]{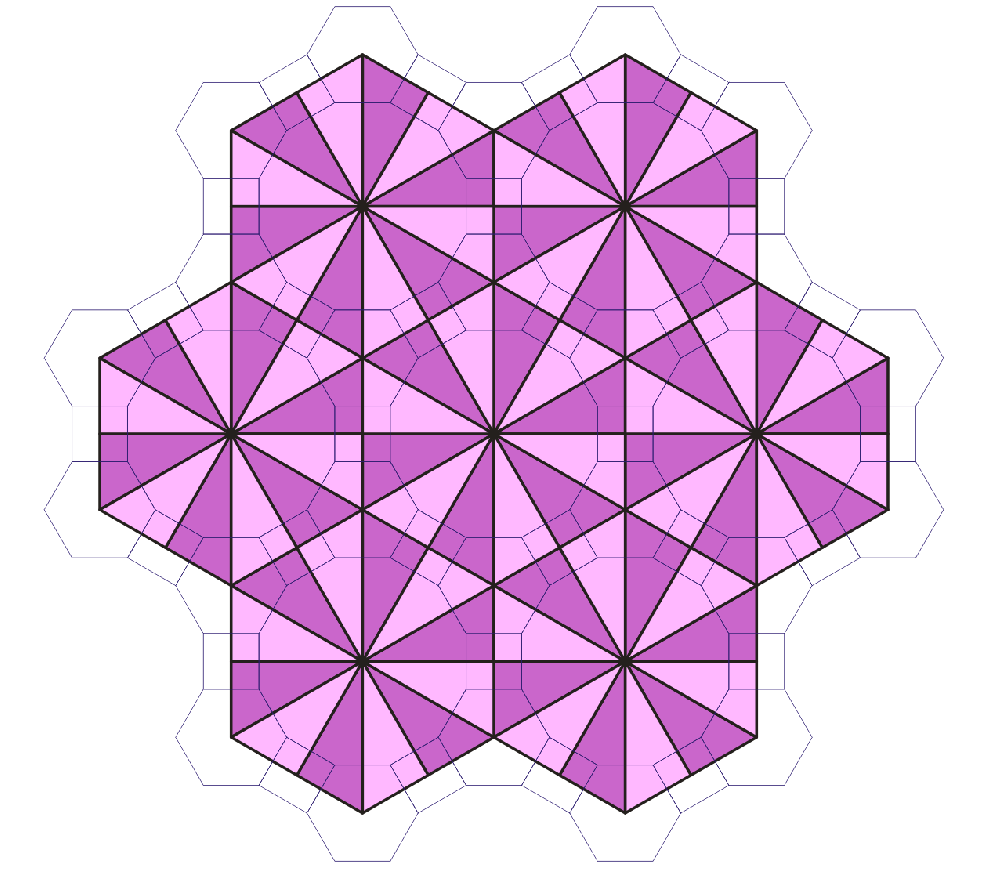} & \includegraphics[width=2in]{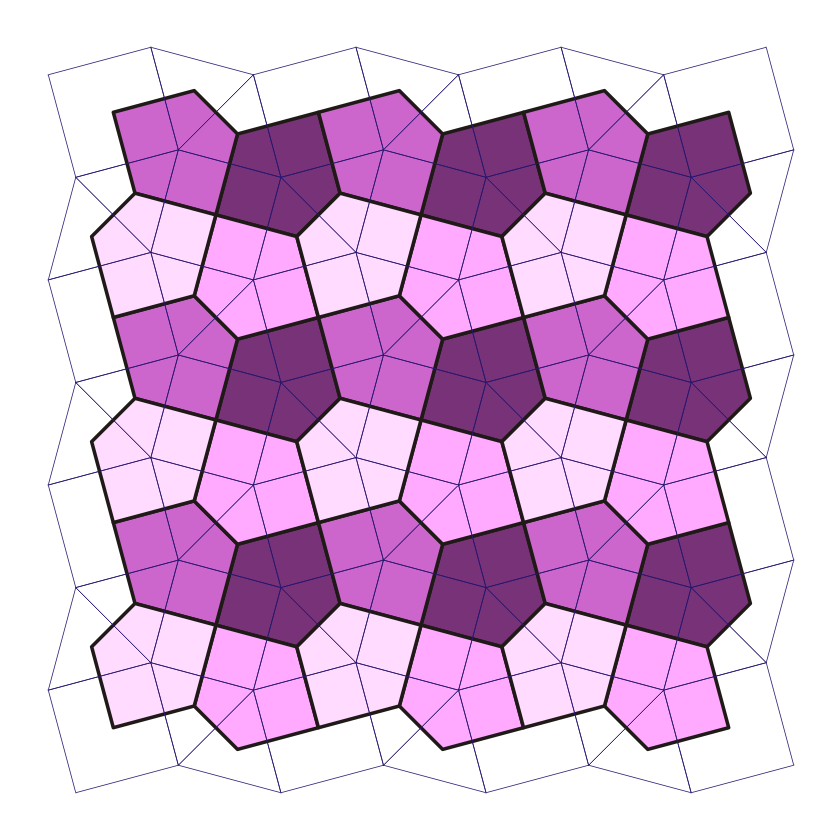} \\
       \hspace*{-0.5in}{\small Lattice graph \#8 } & {\small Lattice graph \#9} & {\small Lattice graphs \#10 and \#14} \\
\end{tabular}
    \caption{Lattice graphs $\G$ and $\G^*$ with orthogonally dual embeddings shown, so faces of $\Gb$ are right kites in each case.
      Labels from Table~\ref{table1}. Figures from~\cite{Cairo-wiki}.}
\label{fig:dual-list}
\end{figure}

\begin{figure}
  \includegraphics[height=1.4in]{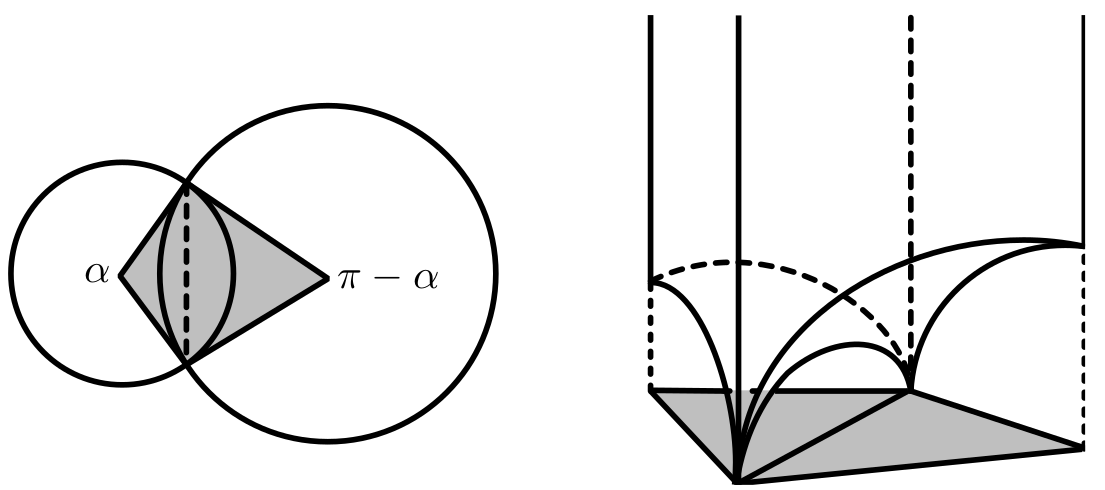} \qquad
  \scalebox{0.8}{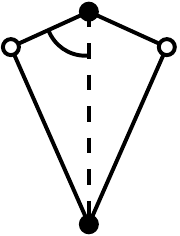} \\
  \hspace*{0.3in} (a) \hspace*{1.5in} (b) \hspace*{1.1in} (c)
\caption{(a) A right kite formed by radii of intersecting orthogonal
  circles.  (b) An ideal hyperbolic polyhedron bounded by vertical
  planes and intersecting hemispheres above the kite, which consists of
  two $3/4$-ideal tetrahedra. (c) As a face of $\gb$, the kite has half-angle $\theta_e$ at edge $e$ of $\gb$.}
\label{fig:pyr-kite}
\end{figure}
\end{center}

\begin{theorem}\label{thm:rak}
  If $G$ and $G^*$ are dual simple $3$--connected graphs on the torus with faces that are topologically disks, then
\begin{enumerate}
\item $G$ and $G^*$ admit an orthogonally dual embedding on the torus, which is unique up to M\"obius transformations,
\item $G$ and $G^*$ are Tait graphs of an alternating link $L$ in $T^2\times I$, such that $(T^2\times I)-L$ is hyperbolic,
\item $\volp(G) \ \leq \ \vol(G) \ \leq \ \volbp(G)+\volbp(G^*)$.
\end{enumerate}
\end{theorem}
  
\begin{proof}
This theorem restates results from hyperbolic knot theory in
\cite[Theorem~7.5]{ckp:gbal}, and we explain the correspondence below.
The proof relies on a result about the existence of orthogonal circle
patterns on the torus, due to Bobenko and Springborn \cite{bs2004}.
  
The conditions on $L$ in \cite[Theorem~7.5]{ckp:gbal} correspond to the following conditions on $G$ and $G^*$: 
\begin{enumerate}
\item[(a)] $G$ and $G^*$ are $2$-connected; in particular, neither $G$ nor $G^*$ has loops,
\item[(b)] the faces of $G$ and of $G^*$ are topologically disks on the torus, and
\item[(c)] if $e_1, e_2\in EG$ form a cut-set for $G$ then a component
  of $EG\setminus\{e_1, e_2\}$ is a path graph, and similarly for $G^*$.
\end{enumerate}
Condition (a) means that $L$ is reduced and weakly prime \cite[Definition~7.1]{ckp:gbal};
(b) means that $L$ has a cellular emebedding on the torus; and
(c) means that $L$ has no cycle of tangles \cite[Definition~6.2]{ckp:gbal}.
Finally, $L$ has no bigons because degree-$2$ vertices of $G$ correspond to bigons of $L$.

Our conditions on $G$ and $G^*$ imply the conditions (a),\, (b),\,
(c).  In particular, since edge-connectivity is at least
vertex-connectivity, $3$--connectedness precludes a two-edge cut as in
(c).  So the conditions in \cite[Theorem~7.5]{ckp:gbal} are satisfied
by our conditions on $G$ and $G^*$, which proves $(1)$ and $(2)$.

Since the projection of $L$ is the medial graph of $G$,
$$ \volbp(L) = \sum_{f \in \{\text{faces of}
  \ L\}}\vol(B_{|f|})=\volbp(G)+\volbp(G^*). $$ It remains to show
that $\volp(L)=\volp(G)$.  When the faces of $\gb$ are right kites,
the white vertices of $\gb$ are the intersections of the orthogonal
circle pattern, which correspond to crossings of $L$.  The crossings
of $L$ are the vertices of the projection graph, and we can
checkerboard color its faces.  Thus, the orthogonal circle pattern
given by inscribed circles of $G$ and of $G^*$ is exactly the
orthogonal circle pattern given by the circumscribed circles of the
shaded and unshaded faces of $L$.  Thus, the right-angled polyhedron
$P(G)$ is exactly $P(L)$ in the theorem, whose $1$-skeleton is the
projection graph of $L$.  This proves (3).
\end{proof}

For example, the $3^3$-$4^2$ lattice graph $\G$ is sometimes drawn as shown at left, but here is how to draw its orthogonally dual embedding:

\begin{center}
\begin{tabular}{ccc}
  \includegraphics[height=0.8in]{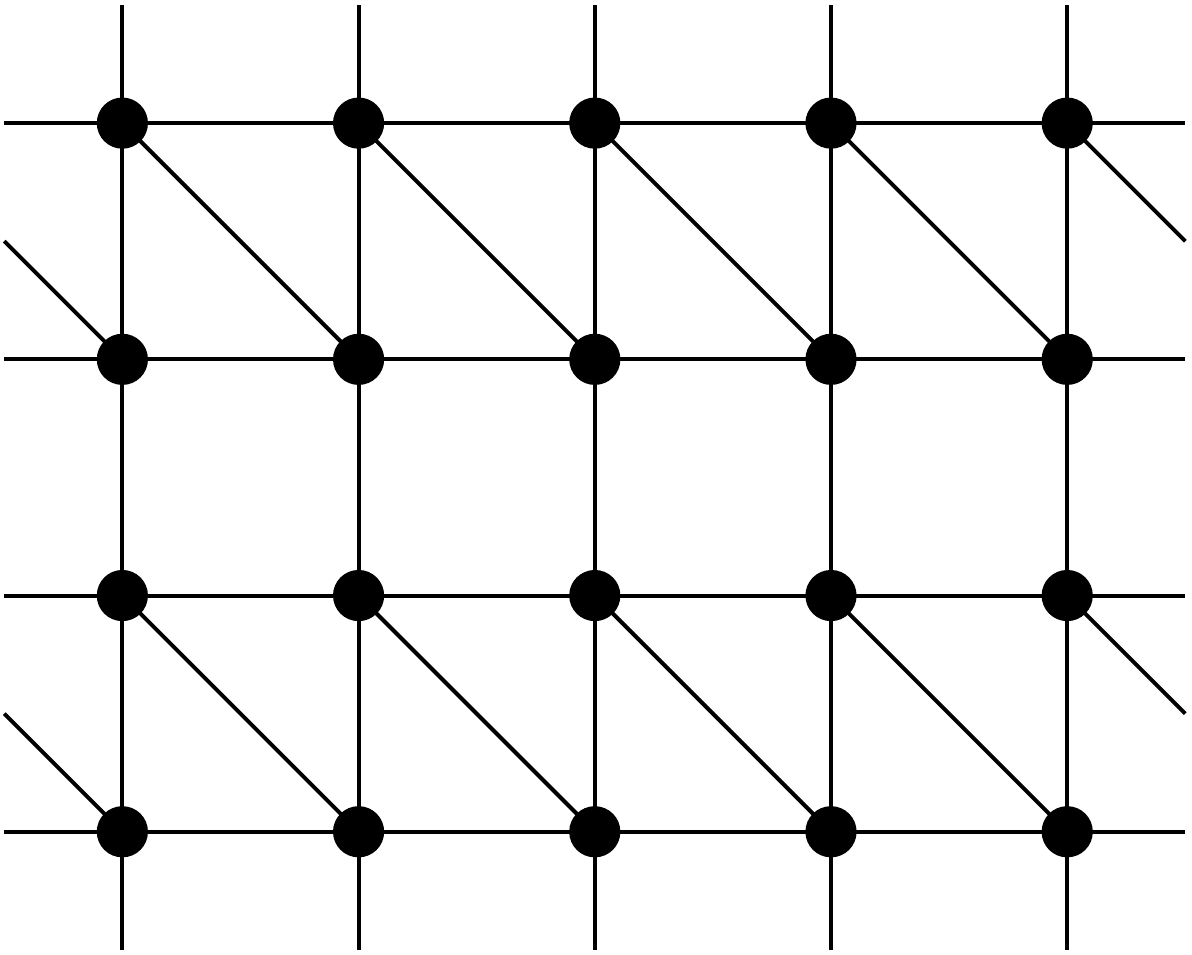} &
\begin{tabular}{c}
  \ $\longrightarrow$\\ \\ \\ \\ \\
 \end{tabular}
    & \includegraphics[height=0.8in]{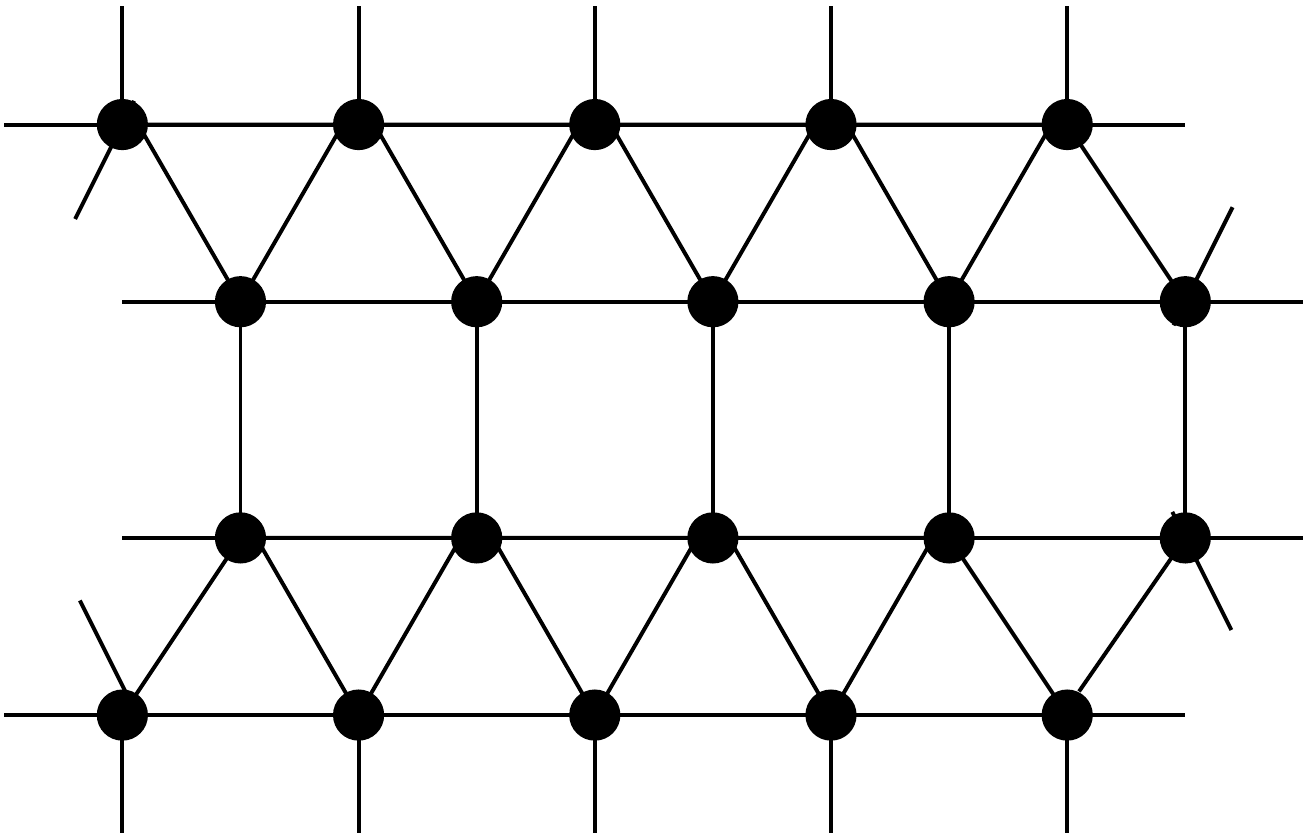}
\end{tabular}
\end{center}
\vspace*{-0.25in}

By Theorem~\ref{thm:rak}, for planar lattice graphs $\G$ that satisfy orthogonal duality,
$$ \volp(\G) \ \leq \ \vol(\G) \ \leq \ \vg.$$
Moreover, as we did for $\vg$, we can compute $\volp(\G)$ using only
the local geometry of its orthogonally dual embedding.  If $G$ and
$G^*$ are orthogonally dual, then the right kite angles of $\gb$ are
uniquely determined.  For $e\in E\gb$, let $2\theta_e$ be the angle
swept out counter-clockwise around the black vertex of $e$ to the
adjacent edge, which is one of the vertex angles of a right kite.
Then $\theta_e$ will be called the half-angle at $e$, as shown in
Figure~\ref{fig:pyr-kite}(c).  The chord shown in
Figure~\ref{fig:pyr-kite}(a) is the diagonal of a right kite joining
the white vertices of $\gb$.  If $v\in VG$ and $v'\in VG^*$ are the
black vertices of the kite, then it has sides $e, e'\in E\gb$, such
that $2\theta_e+2\theta_{e'}=\pi$.  Let $L(\theta)$ be the Lobachevsky
function, as in Definition~\ref{def:bipyramid}.  By Milnor's volume
formula, the volume of the two $3/4$-ideal tetrahedra shown in
Figure~\ref{fig:pyr-kite}(b) is $L(\theta_e)+L(\theta_{e'})$.
Therefore,
\begin{equation}\label{eq:vol-inequality}
   \volp(G)= \sum_{e\in E\gb} 2L(\theta_e).
\end{equation}


For a regular planar lattice graph $\G$, the lower bound in
Conjecture~\ref{conj:vol-entropy} was proved with equality using
number theory in \cite[Theorems~12, 13]{ckl:mm_voldet}.  Below we give
another proof using the toroidal dimer model and the fact that $\G$
satisfies both orthogonal duality and a geometric condition called an
\emph{isoradial embedding.}  Namely, $\G$ is embedded isoradially in
the plane if each face is inscribed in a circle of radius $1$ whose
center is in the interior of that face.  Such an isoradial embedding
is equivalent to a rhombic embedding of an associated quadrangulation
by joining each vertex of $\G$ with the circumcenter of every face
incident to that vertex.  Equivalently, $\gb$ is embedded such that
its black vertices are vertices of the rhombi, and its white vertices
are centers of the rhombi.

\begin{theorem}\label{thm:ortho-iso}
  For the regular planar lattice graphs, $\G_{\triangle},\,\G_{\square},\,\G_{\hexagon}$,
  $$ \volp(G) = \vol(G) = \volbp(G)+\volbp(G^*)= 2\pi\,\m(p(z,w))=2\pi\,\zfd_G.$$
  Thus, the lower bound in Conjecture~\ref{conj:vol-entropy} holds with equality.
\end{theorem}
\begin{proof}
The toroidal dimer model is a statistical mechanics model of the set
of dimer coverings of $\gb=\Gb/\Lambda$.  The characteristic
polynomial of the dimer model is defined as $p(z,w) = \det \K(z,w)$,
where $\K(z,w)$ is the Kasteleyn matrix.
Let $\gb_n=\Gb/(n\Lambda)$ be the finite balanced bipartite toroidal graph, and
$Z(\gb_n)$ be the number of dimer coverings of $\gb_n$.
By \cite{KOS}, its entropy is given by
\begin{equation}\label{eqn:KOS}
  \lim_{n\to\infty}\frac{1}{n^2}\log Z(\gb_n) = \m(p(z,w)).
\end{equation}
See \cite{ckl:mm_voldet} for details, examples and references.

For an isoradial graph, critical edge weights are defined as
$\nu(e)=2\sin\theta_e$, where edge $e$ is the diagonal of a rhombus and $\theta_e$ is the half-angle at $e$.
Thus, $\nu(e)$ is the length of the other diagonal of the rhombus, which is dual to $e$.
The function $\nu$ is called the critical weight function for the
isoradial dimer model on $\gb$, introduced in \cite{kenyonlap}.
By \cite{BdT}, in the isoradial case with the critical edge weights,
the entropy of the toroidal dimer model can be computed using only the
local geometry of the isoradial embedding:
\begin{equation}\label{eqn:BdT}
  2\pi\,\m(p_\nu(z,w)) = 2\pi\lim_{n\to\infty}\frac{1}{n^2}\log Z(\gb_n,\nu) = \sum_{e\in E\gb}\big(2\theta_e\log(2\sin\theta_e)+2L(\theta_e)\big).
\end{equation}
We emphasize two key differences between equations (\ref{eqn:KOS}) and
(\ref{eqn:BdT}): First, edges in equation~(\ref{eqn:KOS}) do not have
the critical edge weights because we are interested in counting
spanning trees.  Second, the Mahler measure $\m(p(z,w))$ in
equation~(\ref{eqn:KOS}) cannot be expressed using Lobachevsky
functions as in equation~(\ref{eqn:BdT}); see \cite{ckl:mm_voldet} for
exact computations of $\m(p(z,w))$.

We now use that each $\G$ satisfies both orthogonal duality and
isoradiality.  Let $v$ be a black vertex of $\gb$; i.e., $v\in VG\cup
VG^*$.  Let $2\theta_i$ be the angle between adjacent edges $e_i,\,
e_{i+1}$ of $\gb$ incident to $v$.  When $\theta_i=\theta$ for all $i$
at some $v\in VG\cup VG^*$, then all edges incident to $v$ have the
same critical weight $2\sin\theta$, which can be factored out of the
Kasteleyn matrix.  This factor appears as the summand
$\log(2\sin\theta)$ in $\m(p_\nu(z,w))$;
i.e., $\m(p_\nu(z,w))=\m(p(z,w))+\log(2\sin\theta)$.
Moreover, in this case $\theta_i=\pi/|v|$ for all $i$, so
   $$ 2\pi\log(2\sin\theta) = 2\pi\,\log\left( \frac{1}{|v|} \sum_{i=1}^{|v|}2\sin\theta_i\right) = \sum_{i=1}^{|v|}2\theta_i\log(2\sin\theta_i)=\sum_{e\in E\gb}2\theta_e\log(2\sin\theta_e).$$
Since this occurs at every $v\in VG\cup VG^*$ for the regular planar lattice graphs, we have
  \begin{eqnarray*}
    2\pi\,\m(p(z,w)) & = & 2\pi\big(\m(p_\nu(z,w)) - \log(2\sin\theta)\big) = 2\pi\,\m(p_\nu(z,w)) - \sum_{e\in E\gb}2\theta_e\log(2\sin\theta_e) \\
    & = & \sum_{e\in E\gb} 2L(\theta_e) = \volp(G).
  \end{eqnarray*}  
The last equality is by equation~(\ref{eq:vol-inequality}).    
Moreover, since $\vol(B_n)=2n\,L(\pi/n)$,
\begin{eqnarray*}
   \sum_{e\in E\gb} 2L(\theta_e)  &=& \sum_{v\in VG\cup VG^*} \sum_{i=1}^{|v|}2L(\theta_i) 
      = \sum_{v\in VG\cup VG^*} 2|v|\,L\big(\frac{\pi}{|v|}\big) = \sum_{v\in VG\cup VG^*}\vol(B_{|v|}) \\
      &=& \volbp(G)+\volbp(G^*).
    \end{eqnarray*}  
Therefore, by Theorem~\ref{thm:rak},
  $$ \volp(G) = \vol(G) = \volbp(G)+\volbp(G^*)= 2\pi\,\m(p(z,w)).$$
Finally, $\m(p(z,w))=\zfd_G$ essentially by equation~(\ref{eqn:KOS}), as explained in \cite[Proposition~5.3]{ck:det_mp}.
When we divide by $|VG|$, the lower bound in Conjecture~\ref{conj:vol-entropy} holds with equality for $\G$.
\end{proof}

\begin{remark}
In \cite{kenyonlap}, Kenyon interpreted the Lobachevsky functions as
the volume of a hyperbolic polyhedron $\P$ associated to $(\Gb)^*$,
using the circle pattern given by circumcircles of the faces of
$(\Gb)^*$.  Vertices of $(\gb)^*$ are the intersections of a circle
pattern whose radii are the edges of rhombi.  When lifted to the
universal cover, considered as the plane at infinity for $\H^3$, this
circle pattern defines a biperiodic ideal hyperbolic polyhedron $\P$
in $\H^3$ whose vertices are the vertices of $(\Gb)^*$.  Let
$\PP=\P/\Lambda$, which is a hyperbolic polyhedron with finite volume.

If $G$ has an isoradial embedding, and $\PP$ is the hyperbolic polyhedron
associated to $(\gb)^*$, then
$$ \vol(\PP) = \sum_{e\in E\gb} 2L(\theta_e)\ \leq \ \volbp(G)+\volbp(G^*).$$
The first equality was proved in \cite{kenyonlap}.  For the
inequality, we can decompose $\PP$ into tetrahedra, and collect them
around each black vertex of $\gb$ to form a bipyramid.  This is
possible because around each vertex $v$, $\sum\theta_e=\pi$ for all
edges $e$ incident to $v$.  By~\cite{Adams:bipyramids}, the maximal
volume of hyperbolic bipyramids is achieved by the regular bipyramids,
whose volumes sum to $\volbp(G)+\volbp(G^*)$.

When $G$ is isoradial, the circles that determine $\PP$ are all
congruent, which is not required to bound a hyperbolic polyhedron.  On
the other hand, the polyhedron $P(G)$ given by an orthogonal circle
pattern is right-angled, which is another special case.  The geometry
for these two types of planar lattice graphs corresponds as follows:

\ 

\hspace*{-0.5cm}
\begin{tabular}{lcc}
  & $\G$ and $\G^*$ satisfy & $\G$ and $\G^*$ satisfy \\
  & orthogonal duality & isoradiality \\
  \hline 
  \rule{0pt}{3ex} Circle pattern: & orthogonal, inscribed circles & isoradial, circumscribed circles \\
  \rule{0pt}{3ex} Local geometry: & $\gb$-faces are right kites & black $V\gb$ form rhombi \\
  \rule{0pt}{3ex} Hyperbolic polyhedron: & vertices of $P$ are white $V\gb$ & vertices of $\PP$ are $V(\gb)^*$ \\
  \rule{0pt}{3ex} Volume of polyhedron: & $\displaystyle \vol(P)=\sum\nolimits_{e\in E\gb} L(\theta_e)$ & $\displaystyle \vol(\PP)=\sum\nolimits_{e\in E\gb} 2L(\theta_e)$ \\
  \rule{0pt}{3ex} Volume inequality: & $\displaystyle 2\vol(P) \leq \volbp(G)+\volbp(G^*)$ & $\displaystyle \vol(\PP) \leq \volbp(G)+\volbp(G^*)$ \\ \\
\end{tabular}

\end{remark}

\bibliographystyle{amsplain} 
\providecommand{\bysame}{\leavevmode\hbox to3em{\hrulefill}\thinspace}
\providecommand{\MR}{\relax\ifhmode\unskip\space\fi MR }
\providecommand{\MRhref}[2]{%
  \href{http://www.ams.org/mathscinet-getitem?mr=#1}{#2}
}
\providecommand{\href}[2]{#2}


\end{document}